\documentclass[11pt,letterpaper]{article}
\usepackage[margin=1in]{geometry}
\usepackage[T1]{fontenc}
\usepackage{lmodern}
\usepackage{amsmath,amsthm}
\usepackage{microtype}
\usepackage{algorithm}
\usepackage{ifpdf}
\usepackage[numbers,sort&compress]{natbib}
\usepackage[hidelinks]{hyperref}
\usepackage[nameinlink,noabbrev]{cleveref}
\newtheorem{theorem}{Theorem}[section]
\newtheorem{lemma}[theorem]{Lemma}
\newtheorem{proposition}[theorem]{Proposition}
\newtheorem{corollary}[theorem]{Corollary}
\newtheorem{assumption}[theorem]{Assumption}
\theoremstyle{definition}
\newtheorem{definition}[theorem]{Definition}
\theoremstyle{remark}
\newtheorem{remark}[theorem]{Remark}
\newtheorem{example}[theorem]{Example}
\newenvironment{keywords}{\par\small\noindent\textbf{Keywords.} }{\par}
\newenvironment{MSCcodes}{\par\small\noindent\textbf{MSC2020.} }{\par}
\date{}
\usepackage{amsfonts,amssymb}
\usepackage{graphicx}

\usepackage{array}
\usepackage{placeins}
\usepackage{needspace}
\usepackage{booktabs}
\usepackage{tikz}
\usetikzlibrary{backgrounds,arrows.meta,calc}
\definecolor{cOne}{HTML}{0072B2}
\definecolor{cTwo}{HTML}{D55E00}
\definecolor{cW}{HTML}{009E73}
\tikzset{every picture/.style={line width=1.6pt,execute at begin picture={\setlength{\arrayrulewidth}{1.2pt}}},
  rootv/.style={rectangle,draw=black!75,inner sep=2pt,font=\small,minimum size=4.6mm},
  vtx/.style={circle,draw=black!75,inner sep=1pt,font=\small,minimum size=4.6mm},
  ed/.style={-{Latex[length=2.2mm,line width=1.2pt]},black!70,line width=1.6pt},
  thk/.style={-{Latex[length=2.4mm,line width=1.2pt]},black,line width=2.6pt},
  blob/.style={line width=1.6pt,draw=black!35,dash pattern=on 2.5pt off 2pt,rounded corners=6pt},
  lb/.style={font=\scriptsize,black!65},
  pl/.style={font=\small}}
\usepackage[noend]{algpseudocode}
\algrenewcommand\algorithmicrequire{\textbf{Input:}}
\algrenewcommand\algorithmicensure{\textbf{Output:}}

\numberwithin{equation}{section}

\crefname{assumption}{Assumption}{Assumptions}
\crefname{remark}{Remark}{Remarks}
\crefname{example}{Example}{Examples}

\newcommand{\Orders}{\mathcal{S}}
\newcommand{\Fset}{\mathcal{F}}
\newcommand{\R}{\mathbb{R}}
\newcommand{\pr}{\mathbb{P}}
\providecommand{\E}{}\renewcommand{\E}{\mathbb{E}}
\DeclareMathOperator{\TV}{TV}
\newcommand{\one}{\mathbf{1}}
\newcommand{\zero}{\mathbf{0}}
\DeclareMathOperator{\diag}{diag}
\DeclareMathOperator{\adj}{adj}

\title{The Order Mass Identity of Rooted Forests, with a Direct Proof of the Extended Aldous-Broder Theorem}

\author{Hugh Zheng\thanks{Department of Statistics, Texas A\&M University (\href{mailto:hugh.zheng@tamu.edu}{\texttt{hugh.zheng@tamu.edu}}).}}

\ifpdf
\hypersetup{
  pdftitle={The Order Mass Identity of Rooted Forests, with a Direct Proof of the Extended Aldous-Broder Theorem},
  pdfauthor={H. Zheng}
}
\fi

\begin{document}

\maketitle

\begin{abstract}
The Aldous-Broder algorithm constructs a spanning tree from the
first-entrance edges of a covering random walk. Its distribution is known,
but deriving it by summing over the orders in which a fixed tree can be
discovered remains difficult. In this paper, we evaluate this sum by
establishing an order-mass identity for rooted forests. For real matrices
with zero row sums, the identity expresses the sum of order weights
associated with each root as a principal cofactor of a matrix on the
roots, under a nonsingularity condition involving only the sets encountered
along the growth orders of the given forest. We prove the identity using
linear algebraic tools. With one root, the sum equals one, giving a direct
proof of the extended Aldous-Broder theorem for every finite irreducible
Markov chain started at a fixed state. With multiple roots, we derive the
forest law produced by the stopped covering walk and show how it differs
from the weight-proportional forest law. The identity and the resulting
tree and forest laws are formalized in Lean~4.
\end{abstract}



\section{Introduction}\label{sec:intro}

Random spanning forests are used in Bayesian clustering and graph signal
processing. In Bayesian clustering, each component of a forest represents
a group of observations, and a distribution over forests induces a
distribution over partitions
\cite{LuoSangMallick2021,DuanRoy2024}. In graph signal processing,
random forests are used to estimate regularized and interpolated signals
by Monte Carlo \cite{PilavciAmblardBarthelmeTremblay2021}.
These applications motivate studying the distribution of the forest
produced by a sampling algorithm.

Random walks give a simple way to generate trees: start at one vertex, run the walk until every vertex has been visited, and retain the edge that for the first time enters a new vertex. This is the famous Aldous-Broder construction \cite{Broder1989,Aldous1990}, and it turns out that the retained edges form a spanning tree whose law is proportional to the tree weight. Wilson's algorithm provides another classical tree sampler, using loop-erased random walks \cite{Wilson1996,ProppWilson1998}. Recent methods also use fast-forwarded walks \cite{TamDunsonDuan2025} to improve sampling efficiency.

Although the distribution of the Aldous-Broder tree is known to be proportional to the tree weight, a direct derivation using basic linear algebraic results and combinatorial arguments remains difficult. Hence, our first question is: for a fixed tree, can we calculate its probability by summing over the orders in which the walk can discover its vertices? We give a positive answer to this question. This calculation also leads us to consider a second question: if there is more than one root, do the forests constructed by retaining all of the first-entrance edges into the nonroot set follow the distribution that is proportional to the forest weights? The answer is no in general and we discuss the resulting forest law.

To formally describe the problem, let $P$ be the transition matrix of an irreducible Markov chain on $V=\{1,\dots,n\}$, $n\ge2$, with stationary distribution $\pi$. Start the chain at the root $r\in V$ and run it until every state has been visited. For each node $v\ne r$, let $f(v)$ be the state from which $v$ was entered for the first time, i.e., the parent of $v$. Then, the edges $\{(f(v),v):v\not=r,v\in V\}$ form a diverging spanning arborescence $F$ pointing away from $r$, called the \emph{first-entrance tree}, with the law that is proportional to the tree weight:
\begin{equation}\label{eq:ab}
\pr(F=F_0)\propto w(F_0)
\end{equation}
where $w(F):=\prod_{v\not= r}\pi_{f(v)}P_{f(v)v}$ denotes the tree weight of $F$ and $\mathcal F_r$ denotes all diverging spanning arborescences rooted at $r$. For simple random walk on a
connected undirected graph with equal edge weights, \eqref{eq:ab} is the uniform tree law. For a general irreducible chain started at a fixed state, \eqref{eq:ab} is the extended Aldous-Broder theorem established in \cite{HuLyonsTang2021} and \cite{FredesMarckert2023}.

The classical proofs \cite{Broder1989,Aldous1990} identify the distribution of the last-exit tree of a stationary walk and then use time reversal to obtain the first-entrance tree law. A weighted reversible version is given in \cite{LyonsPeres2016}. For general irreducible chains, \cite{HuLyonsTang2021} conditions on the number of traversals of each directed edge. Given these counts, the walk corresponds to a uniformly sampled Eulerian path in a directed multigraph. The BEST theorem gives the conditional tree probabilities, and the limiting transition frequencies yield the tree weights in \eqref{eq:ab}. \cite{FredesMarckert2023} gives both an adaptation of the stationary argument and a combinatorial proof that decomposes an encoding of the walk into golf sequences and heaps of cycles.

However, none of the existing proofs take the direct route by evaluating the summation of the joint probabilities of the tree and its compatible growth orders over the orders in which the vertices are discovered. Indeed, in \cite{FredesMarckert2023}, the difficulty of such evaluation is discussed. In this paper, we pursue this calculation using basic algebraic tools.

To make this difficulty precise, let $\tau(X)$ be the order in which the nodes other than $r$ are discovered. For an ordering $\tau=(\tau_1,\dots,\tau_{n-1})$, define $B_0=\{r\}$ and $B_\ell=B_{\ell-1}\cup\{\tau_\ell\}$ for $\ell=1,\ldots,n-1$. For a fixed tree $F_0=\{(f(v),v)\}$, if $f(\tau_\ell)\in B_{\ell-1}$ for every $\ell=1,\ldots,n-1$, then the discovery order $\tau$ is ``compatible" with $F_0$, and we have the factorization (Proposition~\ref{prop:orderlaw})
\begin{equation}\label{eq:obstruction}
\pr\big(F=F_0,\ \tau(X)=\tau\big)\propto w(F_0)\cdot\omega_f(\tau),
\end{equation}
where $\omega_f(\tau)$ is the order weight that depends both on $f$ and $\tau$. The proportionality constant is independent of both $F_0$ and $\tau$. Summing \eqref{eq:obstruction} over the compatible orders gives the probability of $F_0$. Since the tree weight $w(F_0)$ does not depend on the order, it remains to sum the order weights. In particular, the tree law follows once we establish, for each fixed tree,
\[
\sum_{\tau:\tau\text{ compatible with }F_0}\omega_f(\tau)=1.
\]
For different trees, both the number of compatible orders and their weights $\omega_f(\tau)$ can differ substantially, so directly evaluating the summation is not trivial.

To evaluate this sum, we group the compatible orders according to their
first discovered vertex. Once this vertex is treated as an additional
root, the remaining vertices form a smaller problem of the same kind.
This observation motivates us to study rooted forests, even though our
original question concerns a single tree. We establish an identity
that expresses the sum of order weights associated with each root as a
principal cofactor of a matrix indexed by the roots. The proof proceeds
by induction on the number of nonroot vertices, using a cofactor
relation to connect successive steps.

We call this result the \emph{order-mass identity}. For a single root,
the identity gives the required sum of one and hence a direct proof of
the Aldous-Broder theorem. More generally, the sum depends on which
vertices belong to each root's tree, but not on the edges within those
trees. Thus, forests with the same root components have the same order
masses, even though their compatible orders and individual order weights
may differ.

With several roots, the identity also answers our second question.
The stopped covering walk samples each forest with probability
proportional to its forest weight multiplied by a \emph{completion
weight}, which measures the total weight of the edges that can complete
it to a tree rooted at the seed. This extra factor explains why the
stopped walk does not generally sample the weight-proportional forest
law. It also determines whether importance weighting can recover the
target law, since forests with zero completion weight are never sampled.

The article makes four contributions.
\begin{enumerate}
\item We establish the order-mass identity in Theorem~\ref{thm:cofactor}
under a nonsingularity condition involving only the bases along the growth
orders of the given forest and prove it using linear algebraic tools.
\item For a single root, we use the identity to give a direct proof of the
extended Aldous-Broder theorem in Theorem~\ref{thm:ab} and derive the
conditional distribution of the discovery order in Theorem~\ref{thm:orderlaw}.
\item For multiple roots, we derive the forest law produced by the stopped
covering walk in Theorem~\ref{thm:stopped}, making explicit how it differs
from the weight-proportional forest law.
\item We formalize the order-mass identity, the Aldous-Broder theorem,
and the stopped-walk law in Lean~4.
\end{enumerate}

The remainder of the article follows the calculation.
Section~\ref{sec:prelim} introduces harmonic extensions and growth orders.
Section~\ref{sec:walk} gives the joint probability decomposition,
Sections~\ref{sec:identity} and~\ref{sec:proof} establish the order-mass
identity, and Section~\ref{sec:consequences} derives its sampling consequences.
Section~\ref{sec:numerics-main} gives a numerical illustration.
Appendix~\ref{sec:app-proofs} contains the auxiliary proofs.
The supplementary materials contain worked examples, further algebraic
and combinatorial results, formalization details, and the full numerical study.

\section{Preliminaries}\label{sec:prelim}

To express the discovery-order probabilities, we need harmonic
extensions on the sets visited by the walk and orders compatible with
a fixed forest. We introduce these objects in an algebraic and
combinatorial setting, together with the cofactor identities used to
evaluate the resulting sum.

\subsection{Notations and linear algebraic results}\label{sec:notation}
Throughout, $V$ is a set of $n\ge2$ vertices, and
$U^c:=V\setminus U$ for $U\subseteq V$. We use $u,v,w$ for vertices
and $r$ for a root.
For a matrix $A$, we write $A_{\cdot v}$ for its column $v$, $A_{u\cdot}$ for its row $u$,
$A[U]$ for the principal submatrix on the rows and columns in $U$,
$A[R,C]$ for the submatrix on the rows $R$ and the columns $C$, and
$\adj(A)$ for the adjugate of a square matrix, the transpose of its
cofactor matrix. Matrices and vectors are indexed by sets of vertices.
For example, if $U=\{1,2\}$ and $V=\{1,2,3\}$, a matrix
$A\in\R^{U\times U^c}$ has one column, indexed by $3$, so
$A_{\cdot3}$ denotes that column.

The \emph{principal cofactor} of $A\in\R^{U\times U}$
at $u\in U$ is $\det(A[U\setminus\{u\}])$. We use $e_v$ for the
standard basis vectors, $\one$ for the all-ones vector and $\zero$ for
the zero vector. We write $x|_U$ for the restriction of a vector $x$
to $U$, and we read a vector given on a subset of $V$ as a vector on
$V$ with the missing entries $0$ when the context requires it. For a statement $E$, we use the indicator function
\[
1(E):=\begin{cases}1,&\text{if $E$ holds},\\0,&\text{otherwise}.\end{cases}
\]

We use three classical facts about determinants
\cite{HornJohnson2013}: for $A\in\R^{U\times U}$ and
$a,c\in\R^U$,
\begin{equation}\label{eq:adj}
\begin{gathered}
A\adj(A)=\adj(A)A=\det(A)I,\qquad \det(I+ac^{\top})=1+c^{\top}a,\\
\det(A\text{ with column $v$ replaced by }a)=\textstyle\sum_u a_u\adj(A)_{vu}.
\end{gathered}
\end{equation}

The following lemma expresses the adjugate of a matrix with zero row
sums through its principal cofactors and uses them to evaluate a
rank-one perturbation. These relations will be used to compare
cofactors when the root set is enlarged.

\begin{lemma}[Cofactors of a matrix with zero row sums]\label{lem:adj}
Let $S$ be a nonempty finite set, let $A\in\R^{S\times S}$ satisfy
$A\one=\zero$, and let $d\in\R^S$ be its vector of principal
cofactors, $d_u:=\det(A[S\setminus\{u\}])$. Then
\begin{itemize}
\item[(a)] $\adj(A)=\one d^{\top}$ and $d^{\top}A=\zero^{\top}$;
\item[(b)] $\det(A+ac^{\top})=(c^{\top}\one)(a^{\top}d)$ for all
$a,c\in\R^S$.
\end{itemize}
\end{lemma}

\needspace{6\baselineskip}
\subsection{Harmonic extensions}\label{sec:gammadef}

Suppose that values are given on the complement of $U$. A harmonic
extension fills in the values on $U$ so that the corresponding rows
of $L$ annihilate the resulting vector.

\begin{definition}[Harmonic extension]\label{def:gamma}
Let $L\in\R^{V\times V}$, let $U\subsetneq V$ with $\det L[U]\ne0$, and let $m\in\R^{U^c}$. We
call the matrix $\tilde\Gamma^U\in\R^{V\times U^c}$ with the rows
\[
\tilde\Gamma^U[U,U^c]:=-L[U]^{-1}L[U,U^c],\qquad
\tilde\Gamma^U[U^c,U^c]:=I_{|U^c|}
\]
the \emph{harmonic extension map}, and $\gamma^U(m):=\tilde\Gamma^Um\in\R^V$
the \emph{harmonic extension} of $m$ into $U$. In particular,
$\tilde\Gamma^{\emptyset}=I$ and $\gamma^{\emptyset}(m)=m$.
\end{definition}

The following lemma shows that the prescribed values outside $U$
and the equations $(Lx)|_U=\zero$ uniquely determine the extension.
\begin{lemma}\label{lem:harmonic-extension}
Let $L\in\R^{V\times V}$, let $U\subsetneq V$ with $\det L[U]\ne0$, and
let $m\in\R^{U^c}$. For $x\in\R^V$,
\[
x|_{U^c}=m\quad\text{and}\quad(Lx)|_U=\zero
\qquad\Longleftrightarrow\qquad x=\gamma^U(m).
\]
\end{lemma}

\begin{proof}
We have $(Lx)|_U=L[U]\,x|_U+L[U,U^c]\,x|_{U^c}$. Hence, given
$x|_{U^c}=m$, the condition $(Lx)|_U=\zero$ holds if and only if
$x|_U=-L[U]^{-1}L[U,U^c]m$, that is, if and only if $x=\tilde\Gamma^Um$.
This completes the proof.
\end{proof}

The map $m\mapsto\gamma^U(m)=\tilde\Gamma^Um$ is linear. Each column
has a simple interpretation: for $v\in U^c$, the vector
$\gamma^U(e_v)$ equals $1$ at $v$ and $0$ at the other vertices
outside $U$. Its restriction to $U$ solves the Dirichlet problem with
these boundary values. If
$L\one=\zero$, then $\one$ is annihilated by every row of $L$ and equals
$\one$ off $U$. Hence, by uniqueness,
\begin{equation}\label{eq:unitsum}
\gamma^{U}(\one)=\one,\qquad\text{i.e.}\qquad
\sum_{v\in U^c}\gamma^{U}(e_v)=\one .
\end{equation}

For ease of notation, we write
\begin{equation}\label{eq:Gamma}
\Gamma^U:=\tilde\Gamma^U[U,U^c]=-L[U]^{-1}L[U,U^c]\in\R^{U\times U^c}
\end{equation}
for the rows of $\tilde\Gamma^U$ indexed by $U$. If $L\one=\zero$, then
$\Gamma^U\one=\one$ by \eqref{eq:unitsum}.

\subsection{Forests, blocks, and growth orders}\label{sec:combdef}

We specify a rooted forest by recording the parent of each nonroot
vertex. An order is compatible with the forest when each parent is
available before its child is added. We call these compatible orders
growth orders. To record the component membership separately from
the edges within each tree, we use a root indicator whose row for
each nonroot identifies the root of its component.

\needspace{5\baselineskip}
\begin{definition}[Forests, growth orders, root indicator]\label{def:forest}
Let $B\subsetneq V$ be nonempty. A \emph{spanning forest rooted at
$B$} is a forest on vertex set $V$ whose components each contain
exactly one vertex of $B$, with edges oriented away from that vertex.
Isolated roots are allowed. Thus the vertices with no parent are
exactly those in $B$.
\begin{itemize}
\item[(a)] A map $f:B^c\to V$ is a \emph{$B$-parent map} if, for every
$v\in B^c$, some iterate $f^{(q)}(v)$ lies in $B$, where $f^{(0)}(v):=v$
and $f^{(q+1)}(v):=f(f^{(q)}(v))$ as long as $f^{(q)}(v)\in B^c$. The
vertex $f(v)$ is the \emph{parent} of $v$, the elements of
$f^{-1}(u)$ are the \emph{children} of $u$, and the \emph{block} of
$u\in V$ is $D_u:=\{v\in B^c:f^{(q)}(v)=u\text{ for some }q\ge0\}$.
\item[(b)] A \emph{growth order} of a $B$-parent map $f$ is an ordering
$\tau=(\tau_1,\dots,\tau_k)$ of $B^c$ with $f(\tau_\ell)\in B_{\ell-1}$ for
every $\ell$, where $k:=|B^c|$ and
\begin{equation}\label{eq:base}
B_\ell:=B\cup\{\tau_1,\dots,\tau_\ell\}\qquad(0\le \ell\le k)
\end{equation}
is the \emph{base} after $\ell$ steps. We write $\Orders_B(f)$ for the
set of growth orders of $f$.
\item[(c)] The \emph{root indicator} of a $B$-parent map $f$ is the matrix
$M_f\in\{0,1\}^{B^c\times B}$ with entries $(M_f)_{vr}=1$ if
$v\in D_r$ and $0$ otherwise.
\end{itemize}
\end{definition}

We write $\Fset_B$ for the set of spanning forests rooted at $B$,
and $\Fset_r:=\Fset_{\{r\}}$ for the spanning arborescences rooted
at $r$. The following lemma identifies forests with their parent maps.

\begin{lemma}\label{lem:forest-bijection}
Let $B\subsetneq V$ be nonempty. The assignment
$f\mapsto\{(f(v),v):v\in B^c\}$ is a bijection from the $B$-parent
maps to $\Fset_B$. For each $r\in B$, the component rooted at $r$
has vertex set $\{r\}\cup D_r$.
\end{lemma}

\begin{proof}[Proof of Lemma~\ref{lem:forest-bijection}]
For a $B$-parent map, every parent chain terminates at a unique root.
The endpoints of each edge have the same terminal root, and reversing
the parent chain connects each vertex to that root. Hence the
components are precisely $\{r\}\cup D_r$, $r\in B$. Each is connected
and has $|D_r|$ edges, one entering each vertex other than $r$.
It is therefore a tree, with edges oriented away from $r$.
Conversely, a forest in $\Fset_B$ gives each $v\in B^c$ a unique
parent, and following parents reaches the root of its component.
This defines a $B$-parent map. The two constructions are inverse.
\end{proof}

Figure~\ref{fig:example}(a) illustrates the parent map on
$V=\{1,\dots,5\}$ with roots $B=\{1,2\}$ and
$f(3)=1$, $f(4)=3$, $f(5)=2$. The root blocks are
$D_1=\{3,4\}$ and $D_2=\{5\}$, shown by the dashed regions.
The ordering $(3,5,4)$ is a growth order, whereas $(4,3,5)$ is not,
since the parent of $4$ has not yet been added.

Panel (b) represents growth orders as paths between bases, with each
edge adding the vertex on its label. The three paths give the orders
$(3,4,5)$, $(3,5,4)$ and $(5,3,4)$. Equivalently, these are the linear
extensions of the tree order of $f$.

\begin{figure}[H]
\centering
\resizebox{\textwidth}{!}{%
\begin{tikzpicture}[x=1cm,y=1cm,
  ent/.style={vtx,line width=1.6pt},
  bs/.style={rectangle,rounded corners=3pt,draw=black!70,font=\footnotesize,inner sep=2.5pt},
  rw/.style={-{Latex[length=2.2mm,line width=1.2pt]},black!60,dashed,line width=1.6pt},
  lb/.style={font=\footnotesize,black!80},
  tA/.style={draw=cOne,fill=cOne!12},
  tB/.style={draw=cTwo,fill=cTwo!12},
  tW/.style={draw=cW,fill=cW!15},
  eA/.style={ed,cOne},
  eB/.style={ed,cTwo},
  eW/.style={ed,cW}]
\node[rootv,tA] (a1) at (0,0) {$1$};
\node[rootv,tB] (a2) at (1.8,0) {$2$};
\node[ent,tA] (a3) at (0,-0.9) {$3$};
\node[vtx,tA] (a4) at (0,-1.8) {$4$};
\node[ent,tB] (a5) at (1.8,-0.9) {$5$};
\draw[eA] (a1)--(a3); \draw[eA] (a3)--(a4); \draw[eB] (a2)--(a5);
\begin{scope}[on background layer]
\draw[blob,draw=cOne!60] (-0.42,-0.5) rectangle (0.42,-2.22);
\draw[blob,draw=cTwo!60] (1.38,-0.5) rectangle (2.22,-1.32);
\end{scope}
\node[lb,text=cOne] at (0,-2.48) {$D_1$};
\node[lb,text=cTwo] at (1.8,-1.58) {$D_2$};
\node[font=\footnotesize] at (3.75,-0.9) {$M_f=\begin{array}{c|cc}
 & 1 & 2\\\hline 3 & 1 & 0\\ 4 & 1 & 0\\ 5 & 0 & 1\end{array}$};
\node[pl] at (2.0,-3.1) {(a) Forest and root indicator};
\begin{scope}[xshift=6.2cm,yshift=-0.9cm]
\node[bs] (c0) at (0,0) {$\{1,2\}$};
\node[bs] (c3) at (2.1,0.75) {$\{1,2,3\}$};
\node[bs] (c5) at (2.1,-0.75) {$\{1,2,5\}$};
\node[bs] (c34) at (4.5,0.75) {$\{1,2,3,4\}$};
\node[bs] (c35) at (4.5,-0.75) {$\{1,2,3,5\}$};
\node[bs] (cv) at (6.6,0) {$V$};
\draw[eA] (c0)--node[lb,above,text=cOne]{$3$}(c3);
\draw[eB] (c0)--node[lb,below,text=cTwo]{$5$}(c5);
\draw[eA] (c3)--node[lb,above,text=cOne]{$4$}(c34);
\draw[eB] (c3)--node[lb,pos=0.35,right=1pt,text=cTwo]{$5$}(c35);
\draw[eA] (c5)--node[lb,below,text=cOne]{$3$}(c35);
\draw[eB] (c34)--node[lb,above,text=cTwo]{$5$}(cv);
\draw[eA] (c35)--node[lb,below,text=cOne]{$4$}(cv);
\node[pl] at (3.3,-2.2) {(b) Growth orders as paths};
\end{scope}
\end{tikzpicture}}
\caption{Dashed regions mark the root blocks, which exclude the roots themselves.}
\label{fig:example}
\end{figure}
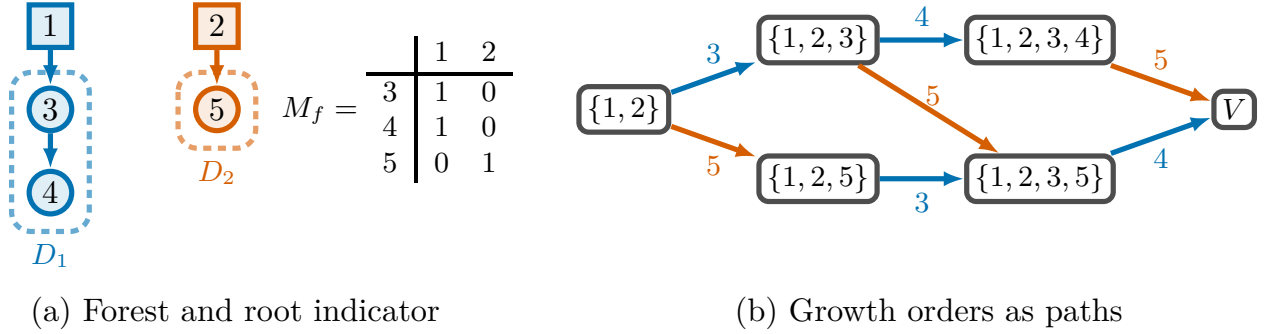

\needspace{5\baselineskip}
We next describe how the blocks decompose along the children of each
vertex.

\begin{lemma}[Blocks]\label{lem:partition}
Let $f$ be a $B$-parent map. The blocks $D_r$, $r\in B$, are pairwise
disjoint and have union $B^c$. Individual blocks may be empty.
For every $u\in V$, the blocks of the children of $u$ are pairwise
disjoint with union $D_u\setminus\{u\}$.
\end{lemma}

As a consequence of Lemma~\ref{lem:partition}, $M_f\one=\one$. Thus
the root indicator records, for each nonroot, the root of its tree,
and nothing else about the forest. We defer the proof of
Lemma~\ref{lem:partition} to Appendix~\ref{sec:app-proofs}.

\section{The covering walk and the sum over discovery orders}\label{sec:walk}

We first calculate the joint probability of a fixed forest and discovery
order by summing over the excursions between successive discoveries.
This calculation separates the forest weight from the factors that depend
on the discovery order and identifies the sum we evaluate in
Section~\ref{sec:identity}.

\subsection{The walk and its first-entrance forest}\label{sec:walkdef}

Let $P$ be an irreducible stochastic matrix on $V$, and let $\pi$ be
its stationary distribution, with $\pi_v>0$, $\one^{\top}\pi=1$ and
$\pi^{\top}P=\pi^{\top}$ \cite{LevinPeres2017}. We use
$(X_t)_{t\ge0}$ to denote the Markov chain with transition matrix $P$.

We stop the walk once it has visited all vertices outside a prescribed
root set and retain the edges by which it first entered those vertices,
as described in Algorithm~\ref{alg:stopped}.

\begin{algorithm}[H]
\caption{Stopped covering walk and its first-entrance forest.}\label{alg:stopped}
\begin{algorithmic}[1]
\Require an irreducible stochastic matrix $P$ on $V$; a nonempty
root set $B\subsetneq V$; a seed $r\in B$.
\Ensure a spanning forest $F_B\in\Fset_B$, each edge $(f(v),v)$
pointing away from the root of its tree.
\State $X\gets r$;\quad $U\gets B$;\quad $F_B\gets\emptyset$
\While{$U\ne V$}
  \State draw $Y$ from the distribution $P_{X\cdot}$
  \If{$Y\notin U$}
    \State $F_B\gets F_B\cup\{(X,Y)\}$;\quad $U\gets U\cup\{Y\}$
  \EndIf
  \State $X\gets Y$
\EndWhile
\State \Return $F_B$
\end{algorithmic}
\end{algorithm}

\needspace{8\baselineskip}
\begin{definition}[Stopped covering walk and discovery order]\label{def:stoppedwalk}
Let $B\subsetneq V$ be nonempty, let $r\in B$ be the \emph{seed}, and
let $(X_t)_{t\ge0}$ be the chain with transition matrix $P$ started
at $X_0=r$. Define $T_v:=\min\{t\ge0:X_t=v\}$ and
$T_B:=\max_{v\in B^c}T_v$. The \emph{stopped covering walk} is
$(X_0,\dots,X_{T_B})$, and its \emph{first-entrance forest} is
\[
F_B:=\{(X_{T_v-1},v):v\in B^c\}.
\]
The \emph{discovery order} $\tau(X)$ is the order in which the walk
first enters the vertices of $B^c$.
\end{definition}

The hitting times and $T_B$ are almost surely finite
\cite{LevinPeres2017}. Following parents strictly decreases first-entrance times until a
root is reached. Thus $v\mapsto X_{T_v-1}$ is a $B$-parent map, and
Lemma~\ref{lem:forest-bijection} gives $F_B\in\Fset_B$. Since the parent
of a vertex is entered before the vertex, $\tau(X)$ is a growth order
of this parent map. For $B=\{r\}$, the construction is the Aldous-Broder algorithm,
and $F_B$ is the first-entrance tree $F$ of \eqref{eq:ab}.

Figure~\ref{fig:walkexample} illustrates a stopped covering walk
$X_0,\dots,X_7$ from the seed $1$, with roots $B=\{1,2\}$.
The retained edges give the forest in Figure~\ref{fig:example}(a), with discovery order
$(3,4,5)$.

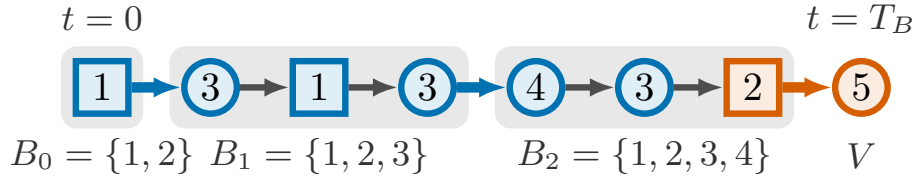
\begin{figure}[H]
\centering
\resizebox{0.75\textwidth}{!}{%
\begin{tikzpicture}[x=1cm,y=1cm,
  ent/.style={vtx,line width=1.6pt},
  bs/.style={rectangle,rounded corners=3pt,draw=black!70,font=\footnotesize,inner sep=2.5pt},
  rw/.style={-{Latex[length=2.2mm,line width=1.2pt]},black!60,dashed,line width=1.6pt},
  lb/.style={font=\footnotesize,black!80},
  tA/.style={draw=cOne,fill=cOne!12},
  tB/.style={draw=cTwo,fill=cTwo!12},
  tW/.style={draw=cW,fill=cW!15},
  eA/.style={ed,cOne},
  eB/.style={ed,cTwo},
  eW/.style={ed,cW}]
\begin{scope}[xshift=0cm,yshift=0cm]
\node[rootv,tA] (w0) at (0,0) {$1$};
\node[vtx,tA] (w1) at (0.95,0) {$3$};
\node[rootv,tA] (w2) at (1.9,0) {$1$};
\node[vtx,tA] (w3) at (2.85,0) {$3$};
\node[vtx,tA] (w4) at (3.8,0) {$4$};
\node[vtx,tA] (w5) at (4.75,0) {$3$};
\node[rootv,tB] (w6) at (5.7,0) {$2$};
\node[vtx,tB] (w7) at (6.65,0) {$5$};
\draw[thk,cOne] (w0)--(w1); \draw[ed] (w1)--(w2); \draw[ed] (w2)--(w3);
\draw[thk,cOne] (w3)--(w4); \draw[ed] (w4)--(w5); \draw[ed] (w5)--(w6);
\draw[thk,cTwo] (w6)--(w7);
\begin{scope}[on background layer]
\fill[black!9,rounded corners=4pt] (-0.36,-0.36) rectangle (0.36,0.36);
\fill[black!9,rounded corners=4pt] (0.59,-0.36) rectangle (3.21,0.36);
\fill[black!9,rounded corners=4pt] (3.44,-0.36) rectangle (6.06,0.36);
\end{scope}
\node[lb] at (0,-0.6) {$B_0=\{1,2\}$};
\node[lb] at (1.9,-0.6) {$B_1=\{1,2,3\}$};
\node[lb] at (4.75,-0.6) {$B_2=\{1,2,3,4\}$};
\node[lb] at (6.65,-0.6) {$V$};
\node[lb] at (0,0.6) {$t=0$};
\node[lb] at (6.65,0.6) {$t=T_B$};
\end{scope}
\end{tikzpicture}}
\caption{Thick steps mark first entrances into nonroots. Gray regions group the excursions between them.}
\label{fig:walkexample}
\end{figure}

To calculate the joint probability of a forest and its discovery order,
we use the harmonic extensions introduced in Section~\ref{sec:gammadef}
for the matrix
\begin{equation}\label{eq:Ldef}
L:=I-P^*,\qquad P^*:=\Pi^{-1}P^{\top}\Pi,\qquad \Pi:=\diag(\pi).
\end{equation}
Here $P^*$ is the transition matrix of the reversed chain. Since $\pi$
is stationary, $L$ has zero row sums. Its proper principal submatrices
are nonsingular by Lemma~\ref{lem:admissible}, so the required harmonic
extensions are well defined.

We collect the factors associated with the forest edges into a forest
weight and the harmonic factors along a growth order into an order
weight. Proposition~\ref{prop:orderlaw} below expresses the joint
probability in terms of these two weights.

\begin{definition}[Forest weight and order weight]\label{def:weight}
Let $f$ be a $B$-parent map, and let $F_0\in\Fset_B$ be its forest.
\begin{itemize}
\item[(a)] The \emph{forest weight} of $F_0$ is
\begin{equation}\label{eq:forestweight}
w(F_0):=\prod_{v\in B^c}\pi_{f(v)}P_{f(v)v}
=\Big(\prod_{v\in B^c}\pi_v\Big)\prod_{v\in B^c}P^{*}_{vf(v)} .
\end{equation}
\item[(b)] More generally, let $L\in\R^{V\times V}$ be any real matrix, and let $\tau\in\Orders_B(f)$
satisfy $\det L[B_\ell]\ne0$ for $0\le \ell\le k-1$. The \emph{order
weight} of $\tau$ is
\begin{equation}\label{eq:weight}
\omega_f(\tau):=\prod_{\ell=1}^{k-1}\tilde\Gamma^{B_{\ell-1}}_{f(\tau_{\ell+1}),\tau_\ell},
\end{equation}
where an empty product is $1$. It depends on $L$ and $f$, but we suppress
$L$ from the notation.
\end{itemize}
\end{definition}

\begin{remark}[Interpretation of the order weight]
For $L$ in \eqref{eq:Ldef}, the harmonic extensions describe exit
probabilities of the reversed chain. To see this, fix $U\subsetneq V$
and $v\in U^c$, and let $x_u$ be the probability of first leaving $U$
at $v$ from $u\in U$. Extend $x$ outside $U$ by $x_v=1$ and
$x_w=0$ for $w\ne v$. First-step conditioning gives
$(Lx)|_U=\zero$, so Lemma~\ref{lem:harmonic-extension} gives
$x=\gamma^U(e_v)$.

Thus, when $f(\tau_{\ell+1})\in B_{\ell-1}$, the factor at step
$\ell$ in \eqref{eq:weight} is the probability that the reversed chain
started at this parent first leaves $B_{\ell-1}$ at $\tau_\ell$.
Otherwise, the parent is $\tau_\ell$, and the factor is one by the
prescribed boundary value.
\end{remark}

\subsection{The order decomposition}\label{sec:orderlaw}

We now compute the joint probability of a fixed forest and discovery
order in terms of the two weights. To express the normalizing
constant, let
\begin{equation}\label{eq:lam}
\lambda_v:=\det\big((I-P)[V\setminus\{v\}]\big)\quad(v\in V),
\qquad \lambda:=(\lambda_v)_{v\in V}.
\end{equation}
The following proposition gives the decomposition for an arbitrary
root set and its simpler form when there is only one root.

\begin{proposition}[Order decomposition]\label{prop:orderlaw}
Let $P$ be irreducible and stochastic with stationary distribution
$\pi$, let $L$ and $\lambda$ be as in \eqref{eq:Ldef} and
\eqref{eq:lam}, let $B\subsetneq V$ be nonempty with seed $r\in B$, and
let $F_B$ and $\tau(X)$ be the first-entrance forest and discovery
order of the stopped covering walk. For every $F_0\in\Fset_B$ with
parent map $f$ and every $\tau\in\Orders_B(f)$,
\begin{equation}\label{eq:orderlaw}
\pr\big(F_B=F_0,\ \tau(X)=\tau\big)
=\frac{\det L[B]}
{\pi_r\big(\prod_{v\in B^c}\pi_v\big)\one^{\top}\lambda}
\big(L[B]^{-1}\big)_{f(\tau_1),\,r}\,w(F_0)\,\omega_f(\tau),
\end{equation}
and $\pr(F_B=F_0,\ \tau(X)=\tau)=0$ for every ordering $\tau$ of $B^c$
that is not a growth order of $f$. In particular, for $B=\{r\}$,
\begin{equation}\label{eq:orderlaw1}
\pr\big(F=F_0,\ \tau(X)=\tau\big)
=\frac{w(F_0)\cdot\omega_f(\tau)}{(\one^{\top}\lambda)\prod_{v\in V}\pi_v}.
\end{equation}
\end{proposition}

The proof sums the excursions between successive first entrances
using a Green function. Reversal expresses its entries through
$L[U]^{-1}$, and normalizing the columns gives harmonic extensions
multiplied by determinant ratios. These ratios cancel across
successive bases, leaving the normalizing factor in the proposition.
The following lemmas supply the identities used in this calculation.

\begin{lemma}[Green function, reversal, normalized columns]\label{lem:admissible}
Let $P$ be irreducible and stochastic with stationary distribution
$\pi$, and let $L$ be the matrix \eqref{eq:Ldef}. Then $L\one=\zero$
and $\det L[U]=\det\big((I-P)[U]\big)>0$ for every nonempty
$U\subsetneq V$. Moreover, for $u,v\in U$,
\begin{equation}\label{eq:green}
\Big(\big((I-P)[U]\big)^{-1}\Big)_{uv}
=\sum_{t\ge0}\big(P[U]^t\big)_{uv}
=\frac{\pi_v}{\pi_u}\big(L[U]^{-1}\big)_{vu},
\end{equation}
and, if $|U|\ge2$, the column $v$ of $L[U]^{-1}$ normalized by its
diagonal entry is a harmonic extension:
\begin{equation}\label{eq:ratio}
\big(L[U]^{-1}\big)_{vv}=\frac{\det L[U\setminus\{v\}]}{\det L[U]},
\qquad
\frac{L[U]^{-1}e_v}{(L[U]^{-1})_{vv}}
=\gamma^{U\setminus\{v\}}(e_v)\big|_U .
\end{equation}
\end{lemma}

The Neumann series gives the Green function, and Cramer's rule gives
the normalized column. We provide the proof in Appendix~\ref{sec:app-proofs}.

\begin{lemma}[Cofactor form of the stationary distribution]\label{lem:stationary}
Let $P$ be irreducible and stochastic with stationary distribution
$\pi$. Then $\lambda_v>0$ for all $v$, and
$\pi=\lambda/(\one^{\top}\lambda)$.
\end{lemma}

\begin{proof}[Proof of Lemma~\ref{lem:stationary}]
Since $I-P$ has zero row sums, Lemma~\ref{lem:adj}(a) gives
$\lambda^{\top}(I-P)=\zero^{\top}$, and each $\lambda_v$ is positive by
Lemma~\ref{lem:admissible}. Hence $\lambda/(\one^{\top}\lambda)$ is a
stationary distribution, and the uniqueness of the stationary
distribution of an irreducible chain \cite{LevinPeres2017}
completes the proof.
\end{proof}

Lemma~\ref{lem:stationary} gives the cofactor formula underlying the
Markov chain tree theorem
\cite{VentselFreidlin1970,LeightonRivest1986,AnantharamTsoucas1989}
without using its expansion into tree sums.

\begin{proof}[Proof of Proposition~\ref{prop:orderlaw}]
An ordering that is not a growth order has probability zero, since
the walk visits a parent before it first enters its child. We therefore
fix $\tau\in\Orders_B(f)$, put $\tau_0:=r$, and let $B_\ell$
be the bases \eqref{eq:base}. Here $B_{\ell-1}$ is the set of vertices that
may be visited before the $\ell$-th first entrance into $B^c$. On the
event $\{F_B=F_0,\tau(X)=\tau\}$, for $\ell=1,\dots,k$, the walk moves
inside $B_{\ell-1}$ from $\tau_{\ell-1}$ to $f(\tau_\ell)$ and then steps
to $\tau_\ell$. By the Markov property, summing over the lengths of these
confined excursions and using \eqref{eq:green}, we have
\[
\pr\big(F_B=F_0,\ \tau(X)=\tau\big)
=\prod_{v\in B^c}P_{f(v)v}\cdot\prod_{\ell=1}^{k}
\Big(\big((I-P)[B_{\ell-1}]\big)^{-1}\Big)_{\tau_{\ell-1},\,f(\tau_\ell)} .
\]
By the reversal in \eqref{eq:green}, the factor $\ell$ equals
$(\pi_{f(\tau_\ell)}/\pi_{\tau_{\ell-1}})(L[B_{\ell-1}]^{-1})_{f(\tau_\ell),\tau_{\ell-1}}$.
Here the numerators complete $\prod_{v\in B^c}P_{f(v)v}$ to the forest weight
$w(F_0)$, and the denominators multiply to
$\prod_{\ell=1}^k\pi_{\tau_{\ell-1}}=\pi_r\big(\prod_{v\in B^c}\pi_v\big)/\pi_{\tau_k}$.
For $2\le \ell\le k$, applying \eqref{eq:ratio} with $U:=B_{\ell-1}$ and
$v:=\tau_{\ell-1}$, so that $U\setminus\{v\}=B_{\ell-2}$, we have
\[
\big(L[B_{\ell-1}]^{-1}\big)_{f(\tau_\ell),\,\tau_{\ell-1}}
=\big(\gamma^{B_{\ell-2}}(e_{\tau_{\ell-1}})\big)_{f(\tau_\ell)}\cdot
\frac{\det L[B_{\ell-2}]}{\det L[B_{\ell-1}]} .
\]
The product of these ratios simplifies to
$\det L[B]/\det L[V\setminus\{\tau_k\}]$, where
\[
\det L[V\setminus\{\tau_k\}]\stackrel{(a)}{=}\lambda_{\tau_k}
\stackrel{(b)}{=}\pi_{\tau_k}\one^{\top}\lambda ,
\]
and $(a)$ and $(b)$ follow from Lemma~\ref{lem:admissible} and
Lemma~\ref{lem:stationary}, respectively. The factor $\ell=1$ is the entry
$(L[B]^{-1})_{f(\tau_1),r}$. Finally, the factor
$\pi_{\tau_k}$ cancels, and the harmonic factors, reindexed by
$\ell\mapsto \ell-1$, form $\omega_f(\tau)$, giving
\eqref{eq:orderlaw}. For $B=\{r\}$, we have $\det L[B]=L_{rr}$
and $(L[B]^{-1})_{rr}=1/L_{rr}$, which gives \eqref{eq:orderlaw1}.
This completes the proof.
\end{proof}

For a single root, summing \eqref{eq:orderlaw1} over the compatible
orders gives the probability of the fixed tree. The forest weight and
normalizing factor are common to these orders, so \eqref{eq:ab}
follows once we establish
\begin{equation}\label{eq:unitmass}
\sum_{\tau\in\Orders_{\{r\}}(f)}\omega_f(\tau)=1
\end{equation}
for every arborescence of positive weight.

\section{The order-mass identity}\label{sec:identity}

To obtain the tree law from the order decomposition, it remains to
evaluate the sum in \eqref{eq:unitmass}. We establish a more general
identity for rooted forests and real matrices $L\in\R^{V\times V}$
with zero row sums. Fix a nonempty root set $B\subsetneq V$, write
$k:=|B^c|$, and let $f$ be a $B$-parent map. The harmonic extensions
in its order weights require the following nonsingularity condition.

\begin{assumption}[Nonsingular bases]\label{ass:hyp}
$L\one=\zero$, and $\det L[B_\ell]\ne0$ for every growth order
$\tau\in\Orders_B(f)$ and every $0\le \ell\le k-1$.
\end{assumption}

\begin{remark}\label{rem:hyp}
Assumption~\ref{ass:hyp} ensures that each Dirichlet problem along the
growth of $f$ has a unique solution. The assumption is satisfied for
every $B$ and $f$
when $L$ is the matrix \eqref{eq:Ldef} of an irreducible chain, by
Lemma~\ref{lem:admissible}, and when $L$ is the Laplacian of a connected undirected
graph with positive weights, whose proper principal submatrices are
positive definite.

With several roots, the first vertex of a growth order can be a child
of different roots. We group the growth orders according to this root
and sum their weights separately.

\begin{definition}[Order mass]\label{def:ordermass}
Under Assumption~\ref{ass:hyp}, the \emph{order mass} of $f$ at the root $r\in B$
is
\begin{equation}\label{eq:g}
g_r(B,f):=\sum_{\tau\in\Orders_B(f):\ f(\tau_1)=r}\omega_f(\tau),
\end{equation}
the total order weight \eqref{eq:weight} of the growth orders whose
first vertex is a child of $r$. We write $g(B,f):=(g_r(B,f))_{r\in B}$.
\end{definition}

\begin{theorem}[Order-mass identity]\label{thm:cofactor}
Let $L\in\R^{V\times V}$, $B$ and $f$ satisfy Assumption~\ref{ass:hyp}. Then,
for every $r\in B$,
\begin{equation}\label{eq:cofactor}
g_r(B,f)
=\det\big((I-\Gamma^BM_f)[B\setminus\{r\}]\big),
\end{equation}
with the determinant of the empty matrix being $1$. In particular, if
$B=\{r\}$, then the order weights of the growth orders of every
arborescence rooted at $r$ sum to $1$.
\end{theorem}

The individual terms on the left of \eqref{eq:cofactor} depend on the
edges within the trees and on the successive bases of each growth
order. The right side, however, depends on the forest only through
its root indicator. Consequently, forests with the same root blocks
have the same order masses, even when their internal tree structures
differ.

\begin{remark}[The chain on the roots]\label{rem:reading}
For a Markov chain, $\Gamma^BM_f$ is stochastic: its entry $(u,u')$
is the probability that the reversed chain started at the root $u$
first leaves $B$ into the tree of $u'$. By Theorem~\ref{thm:cofactor}
and Lemma~\ref{lem:adj}(a), the order masses form an invariant measure
of this chain on the roots. This measure need not be normalized and
may be zero.
Figure~\ref{fig:rootchain} illustrates this chain for the forest in
Figure~\ref{fig:example}.
\end{remark}

\begin{figure}[H]
\centering
\resizebox{0.60\textwidth}{!}{%
\begin{tikzpicture}[x=1cm,y=1cm,
  ent/.style={vtx,line width=1.6pt},
  bs/.style={rectangle,rounded corners=3pt,draw=black!70,font=\footnotesize,inner sep=2.5pt},
  rw/.style={-{Latex[length=2.2mm,line width=1.2pt]},black!60,dashed,line width=1.6pt},
  lb/.style={font=\footnotesize,black!80},
  tA/.style={draw=cOne,fill=cOne!12},
  tB/.style={draw=cTwo,fill=cTwo!12},
  tW/.style={draw=cW,fill=cW!15},
  eA/.style={ed,cOne},
  eB/.style={ed,cTwo},
  eW/.style={ed,cW}]
\begin{scope}[xshift=0cm,yshift=0cm]
\node[rootv,tA] (d1) at (0,0.4) {$1$};
\node[rootv,tB] (d2) at (1.6,0.4) {$2$};
\node[vtx,tA] (d3) at (0,-0.5) {$3$};
\node[vtx,tA] (d4) at (0,-1.4) {$4$};
\node[vtx,tB] (d5) at (1.6,-0.5) {$5$};
\draw[eA] (d1)--(d3); \draw[eA] (d3)--(d4); \draw[eB] (d2)--(d5);
\draw[rw] (d1.north east) to[out=40,in=140] (d2.north west);
\draw[rw] (d2.east) to[out=-20,in=20] (d5.east);
\node[rootv,tA] (s1) at (3.4,-0.4) {$1$};
\node[rootv,tB] (s2) at (5.4,-0.4) {$2$};
\draw[ed] (s1) to[bend left=25] node[lb,above]{$(\Gamma^BM_f)_{12}$} (s2);
\draw[ed] (s2) to[bend left=25] node[lb,below]{$(\Gamma^BM_f)_{21}$} (s1);
\draw[ed] (s1) edge[loop left,min distance=5mm,looseness=8] (s1);
\draw[ed] (s2) edge[loop right,min distance=5mm,looseness=8] (s2);
\end{scope}
\end{tikzpicture}}
\caption{The dashed path starts at root $1$ and first leaves $B$ into the tree of root $2$.}
\label{fig:rootchain}
\end{figure}
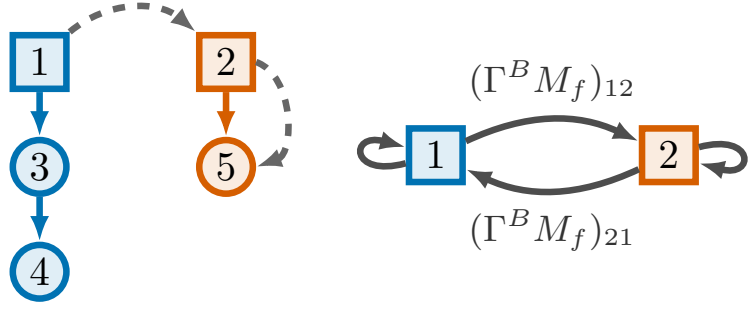
Supplementary Section~\ref{sec:app-comb} gives the forest-sum interpretation and
consequences for successive sampling and linear extensions.

\section{Proof of the order-mass identity}\label{sec:proof}

We prove the identity by induction on the number of nonroot vertices.
For a growth order starting at $w$, we add $w$ to the root set and
remove its edge from its parent. The remaining order is then a growth
order of a forest rooted at $B\cup\{w\}$. The induction hypothesis
expresses its order masses as cofactors on this enlarged root set. To complete the induction, we need to relate these cofactors to those
on $B$.

We first establish the required cofactor relation for two bases
differing by one vertex, then apply it to the forest.

\subsection{The one-vertex transfer lemma}

We begin by comparing harmonic extensions on $U$ and
$U':=U\cup\{w\}$, where $w\in U^c$, $U'\subsetneq V$, and
$\det L[U]\ne0$, $\det L[U']\ne0$. If the prescribed value at $w$ is
zero, the extensions differ by a multiple of $\gamma^U(e_w)$:
\begin{equation}\label{eq:onestep}
\gamma^{U'}(m)=\gamma^{U}(m)+\big(\gamma^{U'}(m)\big)_w\,\gamma^{U}(e_w)
\qquad(m\in\R^{U^c},\ m_w=0),
\end{equation}
where on the left, $m$ is read as a vector on $(U')^c$. Indeed, the
difference of the two sides is annihilated by the rows $U$ of $L$ and
vanishes off $U$, so it is $\zero$ by uniqueness.

To compare the cofactors in \eqref{eq:cofactor}, we write the matrix
there in terms of a base and an aggregation. For $U\subsetneq V$
with $\det L[U]\ne0$ and a matrix $M$ with rows
indexed by $U^c$ and columns by $U$, we write
\begin{equation}\label{eq:N}
N(U,M):=I-\Gamma^UM,\qquad
N(U,M)_{\cdot u}=e_u-\Gamma^UM_{\cdot u}\quad(u\in U).
\end{equation}
When $M\one=\one$ and $L\one=\zero$, the rows of $N(U,M)$ sum to zero,
since $N(U,M)\one=\one-\Gamma^U\one=\zero$. Thus Lemma~\ref{lem:adj}
applies to $N(U,M)$. Taking $U=B$ and $M=M_f$ gives the matrix in
\eqref{eq:cofactor}.

We next specify how the aggregations on $U$ and $U'$ are related.
For a forest, returning $w$ to its parent $r$ merges the block rooted
at $w$ into the block rooted at $r$. We express this relation as an
operation on aggregation matrices.

\begin{definition}[Fold of an aggregation]\label{def:fold}
Let $U\subsetneq V$ be nonempty, $w\in U^c$, $U':=U\cup\{w\}$,
and $r\in U$. Let $M'$ have rows indexed by $(U')^c$, columns indexed
by $U'$, and unit row sums. The \emph{fold} of $M'$ into $r$ is the
matrix $M$ with
rows $U^c$ and columns $U$ obtained by adding the column $w$ of $M'$ to
the column $r$, and inserting $e_r^{\top}$ as the row of $w$:
\begin{equation}\label{eq:fold}
M_{w\cdot}:=e_r^{\top},\qquad
M_{v\cdot}:=M'_{v,U}+M'_{vw}\,e_r^{\top}\quad(v\in (U')^c).
\end{equation}
\end{definition}

The fold preserves unit row sums. For ease of notation, we read every
column of $M'$ as a vector on $U^c$ with the entry $0$ at $w$. Then
\eqref{eq:fold} says that, for $u\in U$,
\begin{equation}\label{eq:foldcol}
M_{\cdot u}=M'_{\cdot u}+1(u=r)\,\big(M'_{\cdot w}+e_w\big).
\end{equation}
Figure~\ref{fig:transfer} illustrates the move and the fold for the
forest in Figure~\ref{fig:example}. The following lemma relates the
principal cofactors before and after the move.

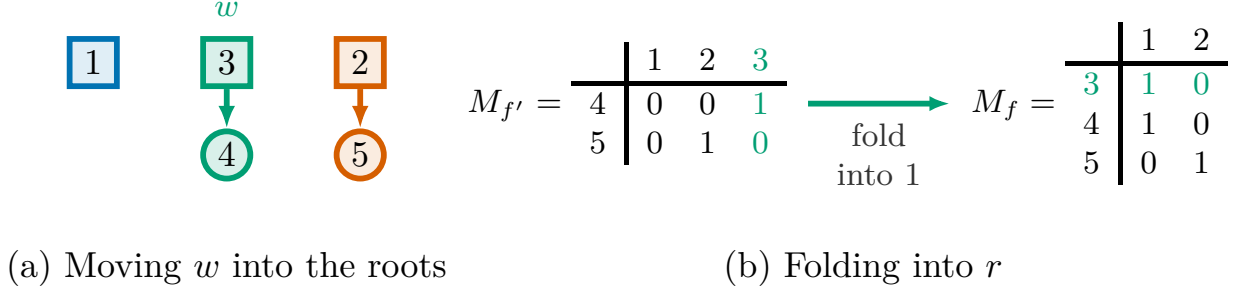
\begin{figure}[H]
\centering
\resizebox{\textwidth}{!}{%
\begin{tikzpicture}[x=1cm,y=1cm,
  ent/.style={vtx,line width=1.6pt},
  bs/.style={rectangle,rounded corners=3pt,draw=black!70,font=\footnotesize,inner sep=2.5pt},
  rw/.style={-{Latex[length=2.2mm,line width=1.2pt]},black!60,dashed,line width=1.6pt},
  lb/.style={font=\footnotesize,black!80},
  tA/.style={draw=cOne,fill=cOne!12},
  tB/.style={draw=cTwo,fill=cTwo!12},
  tW/.style={draw=cW,fill=cW!15},
  eA/.style={ed,cOne},
  eB/.style={ed,cTwo},
  eW/.style={ed,cW}]
\begin{scope}[yshift=0cm]
\node[rootv,tA] (e1) at (0,0.4) {$1$};
\node[rootv,tW] (e3) at (1.3,0.4) {$3$};
\node[rootv,tB] (e2) at (2.6,0.4) {$2$};
\node[vtx,tW] (e4) at (1.3,-0.5) {$4$};
\node[vtx,tB] (e5) at (2.6,-0.5) {$5$};
\draw[eW] (e3)--(e4); \draw[eB] (e2)--(e5);
\node[lb,above=1pt,text=cW] at (e3.north) {$w$};
\node[pl] at (1.3,-1.6) {(a) Moving $w$ into the roots};
\end{scope}
\begin{scope}[xshift=5.2cm,yshift=0cm]
\node[font=\footnotesize] (m1) at (0,0) {$M_{f'}=\begin{array}{c|ccc}
 & 1 & 2 & \color{cW}3\\\hline 4 & 0 & 0 & \color{cW}1\\ 5 & 0 & 1 & \color{cW}0\end{array}$};
\node[font=\footnotesize] (m2) at (4.6,0) {$M_{f}=\begin{array}{c|cc}
 & 1 & 2\\\hline \color{cW}3 & \color{cW}1 & \color{cW}0\\ 4 & 1 & 0\\ 5 & 0 & 1\end{array}$};
\draw[-{Latex[length=2mm,line width=1.2pt]},cW,line width=1.6pt] ($(m1.east)+(0.05,0)$)--node[lb,below=1pt,align=center]{fold\\ into $1$}($(m2.west)+(-0.05,0)$);
\node[pl] at (2.3,-1.6) {(b) Folding into $r$};
\end{scope}
\end{tikzpicture}}
\caption{The transferred vertex and its block are shown in green.}
\label{fig:transfer}
\end{figure}

\begin{lemma}[One-vertex transfer]\label{lem:transfer}
Let $L\in\R^{V\times V}$ satisfy $L\one=\zero$, and let $U\subsetneq V$ be
nonempty with $|U|\le n-2$, $w\in U^c$, $U':=U\cup\{w\}$,
$\det L[U]\ne0$ and $\det L[U']\ne0$. Let $M'$ have rows $(U')^c$,
columns $U'$ and unit row sums, let $r\in U$, and let $M$ be the fold
\eqref{eq:fold} of $M'$ into $r$. Let $d'\in\R^{U'}$ and $d\in\R^{U}$
be the vectors of principal cofactors of $N(U',M')$ and $N(U,M)$. Then
\begin{equation}\label{eq:transfer}
\sum_{u\in U'}\big(\gamma^{U}(e_w)\big)_u\,d'_u
=\sum_{u\in U}\big(\gamma^{U}(M'_{\cdot w}+e_w)\big)_u\,d_u .
\end{equation}
\end{lemma}

We prove the relation by evaluating one auxiliary determinant using
each base.

\begin{proof}[Proof of Lemma~\ref{lem:transfer}]
We write $N':=N(U',M')$ and $N:=N(U,M)$, both of which have zero row sums.
To compare their cofactors, we first extend the columns of $M'$ using
the original base $U$. Put
\[
G:=\Big[\gamma^{U}(M'_{\cdot u})\big|_{U'}\Big]_{u\in U'}\in\R^{U'\times U'},
\qquad
a:=\gamma^{U}(e_w)\big|_{U'}\in\R^{U'} .
\]
To express $N'$ in terms of $G$, fix $u\in U'$. By
\eqref{eq:N}, the column $u$ of $N'$ is
$e_u-\gamma^{U'}(M'_{\cdot u})\big|_{U'}$. Since $M'_{\cdot u}$
vanishes at $w$, \eqref{eq:onestep} applies to $m:=M'_{\cdot u}$, and
its coefficient is $\big(\gamma^{U'}(M'_{\cdot u})\big)_w
=(\Gamma^{U'}M')_{wu}$. Restricting \eqref{eq:onestep} to $U'$, we
obtain
\[
\gamma^{U'}(M'_{\cdot u})\big|_{U'}=G_{\cdot u}+(\Gamma^{U'}M')_{wu}\,a,
\qquad\text{hence}\qquad
N'_{\cdot u}=e_u-G_{\cdot u}-(\Gamma^{U'}M')_{wu}\,a .
\]
Collecting the columns $u\in U'$, we have
\begin{equation}\label{eq:rank1}
N'=(I-G)-a\,(\Gamma^{U'}M')_{w\cdot} .
\end{equation}
Thus $I-G$ is a rank-one perturbation of $N'$. We evaluate its
determinant first through $N'$ and then through $N$.

\emph{Evaluation using the enlarged base.} By \eqref{eq:rank1},
$I-G=N'+a\,(\Gamma^{U'}M')_{w\cdot}$. Since $M'\one=\one$ and
$\Gamma^{U'}\one=\one$, the row $(\Gamma^{U'}M')_{w\cdot}$ sums to one.
Applying Lemma~\ref{lem:adj}(b) to $N'$, we obtain
\begin{equation}\label{eq:Dleft}
\det(I-G)=a^{\top}d'=\sum_{u\in U'}\big(\gamma^{U}(e_w)\big)_u\,d'_u .
\end{equation}

\emph{Evaluation using the original base.} The row $w$ of $G$ vanishes,
since $(\gamma^{U}(M'_{\cdot u}))_w=(M'_{\cdot u})_w=0$ for every
$u\in U'$. Expanding $\det(I-G)$ along this row therefore gives
$\det(I-G)=\det((I-G)[U])$. To express the remaining matrix through
$N$, we apply \eqref{eq:foldcol} and the linearity of $\gamma^U$.
Together with \eqref{eq:N}, these show that its column $u\in U$ is
\[
e_u-\Gamma^UM_{\cdot u}+1(u=r)\,\Gamma^U(M'_{\cdot w}+e_w)
=N_{\cdot u}+1(u=r)\,\Gamma^U(M'_{\cdot w}+e_w),
\]
that is,
\begin{equation}\label{eq:Dblock}
(I-G)[U]=N+\Gamma^{U}(M'_{\cdot w}+e_w)\,e_r^{\top} .
\end{equation}
Since $N\one=\zero$, Lemma~\ref{lem:adj}(b) gives
\[
\det\big((I-G)[U]\big)=\big(\Gamma^{U}(M'_{\cdot w}+e_w)\big)^{\top}d
=\sum_{u\in U}\big(\gamma^{U}(M'_{\cdot w}+e_w)\big)_u\,d_u .
\]
Combining the two evaluations of $\det(I-G)$ gives \eqref{eq:transfer}.
This completes the proof.
\end{proof}

\subsection{Proof of the theorem}

We now apply the transfer lemma to the root indicators of a forest
before and after its first vertex is added to the root set. We first
treat root sets with at least two elements and then use this result
for the single-root case.

\begin{proof}[Proof of Theorem~\ref{thm:cofactor}]
\emph{Step 1: root sets with at least two elements.} We fix $L$ and
induct on $k$, simultaneously for all root sets with at least two
elements and all parent maps satisfying Assumption~\ref{ass:hyp}.
Write $N:=N(B,M_f)$ and $d_u:=\det(N[B\setminus\{u\}])$.
We must prove $g_r(B,f)=d_r$ for each $r\in B$. For this, we will use
$d^{\top}N=\zero^{\top}$ from Lemma~\ref{lem:adj}(a).

For the base case $k=1$, write $B^c=\{w\}$ and $f(w)=:u\in B$.
The unique growth order $(w)$ has weight one, so $g_r(B,f)=1(r=u)$. By
\eqref{eq:unitsum}, $\Gamma^B=\one$ and $\Gamma^BM_f=\one e_{u}^{\top}$, so
$N=I-\one e_{u}^{\top}$ and \eqref{eq:adj} gives
$d_r=1-\sum_{v\in B\setminus\{r\}}(e_{u})_v=1(r=u)$.

For the induction step $k\ge2$, fix $r\in B$ and group the growth orders with
$f(\tau_1)=r$ by their first vertex $w\in f^{-1}(r)$. We put
$B':=B\cup\{w\}$ and let $f'$ be the restriction of $f$ to
$B^c\setminus\{w\}$. Then $f'$ is a $B'$-parent map, and deleting the first vertex gives
a bijection from the growth orders of $f$ starting at $w$ onto
$\Orders_{B'}(f')$. Conversely, every growth order of $f'$ can be
prefixed by $w$, so all its proper bases are bases of $f$.
Thus Assumption~\ref{ass:hyp} holds for $(B',f')$.

For corresponding orders, the definition of the order weight gives
\[
\omega_f(\tau)=(\gamma^{B}(e_w))_{f(\tau_2)}\,\omega_{f'}(\tau_2,\dots,\tau_k),
\]
where the last order weight is taken with respect to $B'$.
We group these shortened orders by $u:=f(\tau_2)\in B'$ and apply
the induction hypothesis to each order mass of $(B',f')$.
Since $(\gamma^{B}(e_w))_u=\Gamma^B_{uw}$ for $u\in B$ and
$(\gamma^{B}(e_w))_w=1$, we obtain
\begin{equation}\label{eq:step1}
g_r(B,f)=\sum_{w\in f^{-1}(r)}\Big(\sum_{u\in B}
\Gamma^B_{uw}\det\big(N'_w[B'\setminus\{u\}]\big)
+\det\big(N'_w[B'\setminus\{w\}]\big)\Big),
\end{equation}
where $N'_w:=N(B',M_{f'})$.

To transfer these cofactors back to $B$, fix $w$. The root blocks of $(B',f')$ are $D_u$ for $u\in B\setminus\{r\}$,
$D_r\setminus D_w$ for $r$, and $D_w\setminus\{w\}$ for $w$. Hence the
fold of $M_{f'}$ into $r$ is $M_f$. Moreover, the column $w$ of
$M_{f'}$ is the indicator of $D_w\setminus\{w\}$, so $M_{f'\cdot w}+e_w$
is the indicator $\one_{D_w}$ of $D_w$. Since $k\ge2$, we have
$|B|\le n-2$. Both $B$ and $B'$ are proper bases of $f$, so the
required minors are nonzero by Assumption~\ref{ass:hyp}.
We can therefore apply Lemma~\ref{lem:transfer} with $U:=B$,
$U'=B'$ and $M':=M_{f'}$ to express the inner sum of
\eqref{eq:step1} as $\sum_{u\in B}(\Gamma^B\one_{D_w})_u d_u$.

It remains to sum these contributions over the children of $r$.
By Lemma~\ref{lem:partition}, their blocks $D_w$ partition $D_r$, and
$\Gamma^B\one_{D_r}=(\Gamma^BM_f)_{\cdot r}=e_r-N_{\cdot r}$. Summing over
$w\in f^{-1}(r)$, we therefore have
\[
g_r(B,f)=\sum_{u\in B}(e_r-N_{\cdot r})_ud_u=d_r-(d^{\top}N)_r=d_r ,
\]
where the last equality holds since $d^{\top}N=\zero^{\top}$.

\emph{Step 2: a single root.} Let $B=\{r\}$. The only principal
cofactor of $N(B,M_f)$ is the empty determinant $1$, so we must show
that the order weights sum to one. If $k=1$, the unique growth order
has the empty product $1$.

For $k\ge2$, fix a child $w$ of $r$, put $B':=\{r,w\}$, and let
$f'$ be the restriction of $f$ to $B^c\setminus\{w\}$. As in Step~1,
$(B',f')$ satisfies Assumption~\ref{ass:hyp}, and deleting the first
vertex from orders starting at $w$ gives precisely its growth orders.
Applying the result for two roots to their order
masses, the sum of $\omega_f(\tau)$ over orders starting at $w$ is
the left side of \eqref{eq:transfer} with $U:=\{r\}$ and $M':=M_{f'}$.
The fold is again $M_f$, and $M_{f'\cdot w}+e_w=\one_{D_w}$.
Since $k\ge2$, the size condition $|U|\le n-2$ holds, and both
$U$ and $U'$ are bases of $f$, so Lemma~\ref{lem:transfer} applies.
Its right side is $\Gamma^{\{r\}}\one_{D_w}$, because $d_r=1$.

The blocks of the children of
$r$ partition $D_r=B^c$ by Lemma~\ref{lem:partition}. Summing over $w$ and
using $\Gamma^{\{r\}}\one=1$, we conclude that
\[
\sum_{\tau\in\Orders_{\{r\}}(f)}\omega_f(\tau)=\Gamma^{\{r\}}\one=1 .
\]
This completes the proof.
\end{proof}

\section{Consequences for the covering walk}\label{sec:consequences}

We return to the Markov chain and take $L$ as in \eqref{eq:Ldef}.
Throughout this section, $P$ is irreducible and stochastic, $\pi$ is
its stationary distribution, and $\lambda$ is the cofactor vector
\eqref{eq:lam}. We first derive the tree law and the conditional law
of its discovery order, then consider the forest obtained by stopping
when all nonroot vertices have been visited. Individual transition
probabilities may be zero. Supplementary Example~\ref{ex:oneway} illustrates these
laws on a chain with one-way support.

\subsection{The Aldous-Broder theorem}

The single-root identity evaluates the sum in
\eqref{eq:unitmass}. Substituting it into the order decomposition gives
the Aldous-Broder theorem in the generality of
\cite{HuLyonsTang2021}.

\begin{theorem}[Aldous-Broder for general chains {\cite[Thm.~3.4]{HuLyonsTang2021}}]\label{thm:ab}
Let $r\in V$ and let $F$ be the first-entrance tree of the covering
walk started at $r$. For every $F_0\in\Fset_r$ with $\{r\}$-parent map $f$,
\[
\pr(F=F_0)=\frac{w(F_0)}{(\one^{\top}\lambda)\prod_{v\in V}\pi_v}
=\frac{w(F_0)}{\sum_{F'\in\Fset_r}w(F')},
\qquad w(F_0)=\prod_{v\ne r}\pi_{f(v)}P_{f(v)v} .
\]
\end{theorem}

\begin{proof}
By Lemma~\ref{lem:admissible}, Assumption~\ref{ass:hyp} holds for every
arborescence. Summing \eqref{eq:orderlaw1} over
$\tau\in\Orders_{\{r\}}(f)$ and applying the single-root case of
Theorem~\ref{thm:cofactor}, we obtain the first form. Since the probabilities
sum to one over $F_0\in\Fset_r$, their common denominator equals the
sum of the numerators, which gives the second form.
\end{proof}

The normalizing constant also yields the arborescence form of the
Markov chain tree theorem for $P^{*}$ by summing over $F_0$.
See \cite{KruckmanGreenwaldWicks2010} for a direct linear-algebra proof
of that theorem. Supplementary Example~\ref{ex:ab} evaluates the tree law on the four-state
chain.

\subsection{The order in which the tree is discovered}

Having obtained the tree law, we can determine how a fixed tree is
discovered. For a tree of positive weight, dividing its joint
probability with a discovery order by its marginal probability
cancels the tree weight and the normalizing constant. The remaining
factor $\omega_f(\tau)$ therefore gives the conditional probability
of the order.

\begin{theorem}[Discovery order given the tree]\label{thm:orderlaw}
Let $r\in V$, let $F$ and $\tau(X)$ be the first-entrance tree and
discovery order of the covering walk started at $r$, and let
$F_0\in\Fset_r$ have parent map $f$ and positive weight $w(F_0)>0$. Then
\[
\pr\big(\tau(X)=\tau\mid F=F_0\big)=\omega_f(\tau)
\qquad(\tau\in\Orders_{\{r\}}(f)),
\]
and $\pr(\tau(X)=\tau\mid F=F_0)=0$ for every other ordering $\tau$.
\end{theorem}

\begin{proof}
By Theorem~\ref{thm:ab}, we have $\pr(F=F_0)>0$. Dividing \eqref{eq:orderlaw1}
by the first form of Theorem~\ref{thm:ab} cancels the tree weight and the
constant. For the orderings that are not growth orders, the claim
follows from Proposition~\ref{prop:orderlaw}.
\end{proof}

\begin{example}[A nonuniform law]\label{rem:nonuniform}
The law need not be uniform on the growth orders. For example, for
simple random walk on the complete graph, the conditional probability
of $\tau$ is the summand of \eqref{eq:Kn} divided by $(n-1)!$, since the
matrix $L$ is then a multiple of the Laplacian of the complete graph,
and harmonic extensions are unchanged by scaling.
\end{example}

\subsection{The stopped covering walk}

With several roots, the order masses need not equal one. To obtain
the stopped-forest law, we relate it to the full tree law. Almost
surely, $F_B$ is the restriction of the full first-entrance tree $F$
of Theorem~\ref{thm:ab} to the edges into $B^c$, since those edges
are determined by time $T_B$. For $F_0\in\Fset_B$, each nonroot
already has its unique incoming edge in $F_0$. Hence $F_B=F_0$ if and
only if $F\supseteq F_0$, and we can obtain its probability by
summing over the full trees containing $F_0$.

To evaluate this sum, we separate the weights of the added edges
from the weight of the fixed forest.

\needspace{11\baselineskip}
\begin{definition}[Completion and completion weight]\label{def:completion}
Let $B\subsetneq V$ be nonempty, $r\in B$, and $F_0\in\Fset_B$.
A \emph{completion} of $F_0$ at $r$ is an arborescence
$F'\in\Fset_r$ containing $F_0$. It adds exactly one edge into each
vertex of $B\setminus\{r\}$ and no other edges. The \emph{completion
weight} of $F_0$ at $r$ is
\begin{equation}\label{eq:phi-def}
\varphi_r(F_0):=\sum_{F'\in\Fset_r:\ F'\supseteq F_0}\
\prod_{u\in B\setminus\{r\}}\pi_{f_{F'}(u)}P_{f_{F'}(u)\,u},
\end{equation}
where $f_{F'}$ is the parent map of $F'$.
\end{definition}

Since the edges of $F_0$ and
the added edges are disjoint,
$\sum_{F'\in\Fset_r:\ F'\supseteq F_0}w(F')=w(F_0)\,\varphi_r(F_0)$.
The following theorem evaluates this completion weight through the
order masses of $F_0$.

\begin{theorem}[Law of the stopped covering walk]\label{thm:stopped}
Let $B\subsetneq V$ be nonempty with seed $r\in B$, and let $F_B$ be
the first-entrance forest of the stopped covering walk. For every
$F_0\in\Fset_B$ with parent map $f$,
\begin{equation}\label{eq:closed}
\pr(F_B=F_0)=\frac{\sum_{F'\in\Fset_r:\ F'\supseteq F_0}w(F')}{\sum_{F'\in\Fset_r}w(F')}=\frac{w(F_0)\,\varphi_r(F_0)}{(\one^{\top}\lambda)\prod_{v\in V}\pi_v},
\end{equation}
and the completion weight is
\begin{equation}\label{eq:phi}
\varphi_r(F_0)=\det L[B]\cdot\prod_{u\in B\setminus\{r\}}\pi_u\cdot
g(B,f)^{\top}L[B]^{-1}e_r .
\end{equation}
\end{theorem}

\begin{proof}
We first prove \eqref{eq:closed}. Since $F_B$ is the restriction of
$F$, summing Theorem~\ref{thm:ab} over the completions of $F_0$ gives
\eqref{eq:closed}.

To evaluate the completion weight, we compute the same probability
from the order decomposition. Summing \eqref{eq:orderlaw} over
$\tau\in\Orders_B(f)$ and grouping by $u:=f(\tau_1)$ gives, by
\eqref{eq:g}, the factor
$\sum_{u\in B}g_u(B,f)(L[B]^{-1})_{ur}=g(B,f)^{\top}L[B]^{-1}e_r$, so
that
\[
\pr(F_B=F_0)=\frac{\det L[B]\; w(F_0)\; g(B,f)^{\top}L[B]^{-1}e_r}
{\pi_r\big(\prod_{v\in B^c}\pi_v\big)\one^{\top}\lambda} .
\]
Comparing with \eqref{eq:closed} and canceling $w(F_0)$ gives
\eqref{eq:phi} when $w(F_0)>0$.

To extend the identity to forests of zero weight, we replace $P$ by $P_\varepsilon:=(1-\varepsilon)P+\varepsilon\one\one^{\top}/n$,
$0<\varepsilon<1$, for which every forest has positive weight. By
Lemma~\ref{lem:stationary}, the stationary distribution of $P_\varepsilon$
is a continuous function of the principal cofactors of
$I-P_\varepsilon$ and converges to $\pi$. Moreover, the minors entering
the harmonic extensions and $L[B]^{-1}$ remain nonzero at the limit
by Lemma~\ref{lem:admissible}. Therefore, both sides of \eqref{eq:phi} are
continuous in $\varepsilon$, and letting $\varepsilon\downarrow0$ proves
\eqref{eq:phi} for every forest. This completes the proof.
\end{proof}

For $B=\{r\}$, the only completion is $F_0$ itself with no added
edge, so $\varphi_r(F_0)=1$ and \eqref{eq:closed} reduces to
Theorem~\ref{thm:ab}. With several roots, the completion weight can
vary between forests. Its dependence on $F_0$ is through the root
indicator $M_f$, as we can see by expressing \eqref{eq:phi} as a
single cofactor.

Since \mbox{$L[B]+L[B,B^c]M_f=L[B]\,(I-\Gamma^BM_f)$} and
$\adj(L[B])=\det(L[B])\,L[B]^{-1}$, the product rule
$\adj(XY)=\adj(Y)\adj(X)$ for adjugates and
Lemma~\ref{lem:adj}(a) give
\begin{equation}\label{eq:phiN}
\varphi_r(F_0)=\Big(\prod_{u\in B\setminus\{r\}}\pi_u\Big)
\det\big((L[B]+L[B,B^c]M_f)[B\setminus\{r\}]\big),
\end{equation}
Thus the completion weight can be evaluated from a determinant of
size $|B|-1$. Supplementary Section~\ref{sec:sm-completion} gives an alternative
derivation by contraction and the computational cost. We tabulate
the resulting law for the four-state example in
Supplementary Section~\ref{sec:app-examples}.

Conditioning the joint order decomposition on the stopped forest
gives its discovery-order law. Unlike the single-root case, the law
also depends on the seed through a column of $L[B]^{-1}$.

\begin{corollary}[Discovery order given the stopped forest]\label{cor:stoppedorder}
Under the assumptions of Theorem~\ref{thm:stopped}, let $F_0\in\Fset_B$ have
parent map $f$ and $\pr(F_B=F_0)>0$. Then, for every
$\tau\in\Orders_B(f)$,
\[
\pr\big(\tau(X)=\tau\mid F_B=F_0\big)
=\frac{\big(L[B]^{-1}\big)_{f(\tau_1),\,r}\,\omega_f(\tau)}
{\sum_{u\in B}g_u(B,f)\big(L[B]^{-1}\big)_{ur}},
\]
and $\pr(\tau(X)=\tau\mid F_B=F_0)=0$ for every other ordering $\tau$ of
$B^c$. In particular, the growth orders entering at $u\in B$ carry the
conditional mass
\[
\frac{g_u(B,f)\big(L[B]^{-1}\big)_{ur}}
{\sum_{u'\in B}g_{u'}(B,f)\big(L[B]^{-1}\big)_{u'r}} .
\]
\end{corollary}

\begin{proof}
Dividing \eqref{eq:orderlaw} by the expression of $\pr(F_B=F_0)$ in the
proof of Theorem~\ref{thm:stopped} cancels the forest weight and the constant. The denominator is nonzero since $\pr(F_B=F_0)>0$. For the orderings
that are not growth orders, the claim follows from Proposition~\ref{prop:orderlaw}.
This completes the proof.
\end{proof}

For $B=\{r\}$, the factor $(L[B]^{-1})_{rr}$ cancels, and
Corollary~\ref{cor:stoppedorder} reduces to Theorem~\ref{thm:orderlaw}.

Wilson's algorithm run with $P^{*}$ and the roots $B$ absorbing samples
forests with law proportional to $w(F_B)$
\cite{Wilson1996,ProppWilson1998,AvenaGaudilliere2018forests,Marchal2000}.
The stopped covering walk instead samples with law proportional to
$w(F_B)\varphi_r(F_B)$. Randomizing the seed does not in general
remove this extra factor, as the counterexample in
Supplementary Section~\ref{sec:sm-seed} shows.
\end{remark}

\begin{remark}[When the tilt vanishes]\label{rem:unbiased}
By \eqref{eq:closed}, the walk from the seed $r$ samples
exactly the law proportional to $w(F_B)$ if and only if the completion
weight $\varphi_r(F_0)$ is the same for all forests $F_0$ of positive
weight. In particular, this holds when every arc from $B^c$ into $B$
ends at $r$. To see this, note that in this case every completion of
positive weight attaches the roots in $B\setminus\{r\}$ through arcs
inside $B$, so $\varphi_r(F_0)$ does not depend on $F_0$. For a reversible
chain, this condition means that every path between $B^c$ and
$B\setminus\{r\}$ passes through $r$. For a nonreversible
chain, the condition can hold without this undirected separation, as
Supplementary Example~\ref{ex:oneway} shows.

\end{remark}

When the completion weight varies, we consider inverse weighting.

\begin{remark}[Importance correction]\label{rem:importance}
Write $q_W(F_0)=w(F_0)/Z_W$ for the weight-proportional law we target,
$q_A(F_0)=w(F_0)\varphi_r(F_0)/Z_A$ for the law \eqref{eq:closed} we
sample forests from, and
$S_r:=\{F_0\in\Fset_B:\ w(F_0)>0,\ \varphi_r(F_0)>0\}$ for the support of
$q_A$. Then, for every function $h$ on $\Fset_B$,
\[
\frac{\E_{q_A}\big[h(F_B)/\varphi_r(F_B)\big]}{\E_{q_A}\big[1/\varphi_r(F_B)\big]}
=\E_{q_W}\big[h(F_B)\mid F_B\in S_r\big].
\]
Thus inverse-completion weighting recovers the target expectation
conditional on the reachable forests. It recovers $\E_{q_W}[h]$ for
every function $h$ if and only if $q_W(S_r)=1$. Any law
$\widetilde q$ supported on $S_r$ satisfies
$\TV(\widetilde q,q_W)\ge 1-q_W(S_r)$, so no reweighting can remove
the discrepancy caused by lost target mass.

For independent draws from $q_A$, the self-normalized sample ratio
is consistent for the conditional target. Its precision also depends
on the variability of the inverse completion weights.
Section~\ref{sec:numerics-main} illustrates both support loss and the
effect of weighting on estimation.
\end{remark}

\section{Numerical illustration}\label{sec:numerics-main}

We illustrate the difference between the stopped-walk and weight-proportional
forest laws by enumerating forests on seven vertices. We use $30$ chains
from each of three families: dense chains, random walks on connected
undirected graphs, and directed cycles with drift. The root sets are
$B=\{1,\ldots,b\}$ for $b=1,\ldots,6$, with seed $1$.
The drift chains share one forest law and are displayed once.
Supplementary Section~\ref{sec:numerics} gives the construction,
sensitivity analyses, and computational checks.

\begin{figure}[!htbp]
\centering
\includegraphics[width=0.85\textwidth]{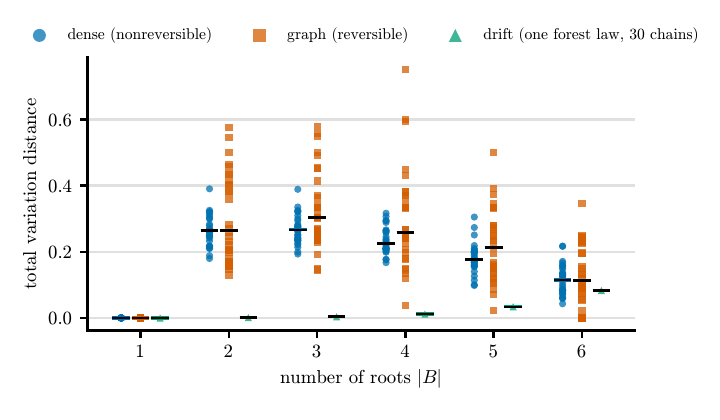}
\caption{Total variation distance between the stopped-walk forest law
with seed $1$ and the weight-proportional law, by root count and chain
family. Points represent chains, with the common drift law shown once. Horizontal bars mark the medians.}
\label{fig:bias}
\end{figure}

Figure~\ref{fig:bias} shows the total variation distance between the two
laws. They agree for a single root. With two roots, the median distance
is about $0.26$ for both the dense and graph families. The largest
distance is $0.75$, attained by a graph chain with four roots, and is
entirely due to unreachable target mass. Thus substantial discrepancies
can occur even for reversible chains, and reweighting cannot always
remove them.

For an example with complete support, consider the first dense chain
with $B=\{1,2\}$ and the size of the seed's tree. Its expectation is
$3.237$ under the target law and $4.359$ under the stopped-walk law.
The self-normalized correction converges to $3.237$. For the graph
chain with the largest lost mass at two roots, the target expectation
is $1.923$, but the correction converges to the conditional expectation
$3.182$. These small examples illustrate the support condition in
Remark~\ref{rem:importance} rather than performance on large graphs.

\section{Discussion}\label{sec:discussion}

The order-mass identity evaluates the sum over growth orders of a
fixed forest as a cofactor on its roots. This connects the direct
calculation of the Aldous-Broder law with the evaluation of
completion weights for the stopped covering walk. The resulting
forest and conditional order laws suggest several further questions.

First, Theorem~\ref{thm:orderlaw} gives the probability of each growth
order of a fixed tree, but its marginals at later positions and the
cost of sampling from this conditional law remain to be analyzed.
Supplementary Corollary~\ref{cor:PL} gives a weighted count of linear extensions
for one choice of $L$. Other choices may yield further weighted
counts with a combinatorial interpretation.

Second, a bijective proof of Theorem~\ref{thm:cofactor} could explain
the equality between the total weights of growth orders and the
forest sums in Supplementary Remark~\ref{rem:forest}. The denominators
$\prod_\ell\det L[B_{\ell-1}]$ depend on the entire order, so a
weight-preserving correspondence would have to account for the
changing bases. Connections with loop erasure, such as those in
\cite{Fomin2001}, may be useful for this question.

Third, the stopped-walk law separates the question of recovering
every target expectation from that of estimating it accurately.
Complete target support is needed for the former, while variability of
the importance weights affects the latter. The numerical examples
illustrate support loss and weight variability separately. It remains
to develop walk-based forest samplers that address both issues and
to determine their computational cost on graphs used in clustering
and signal processing.

The order-mass identity and the tree and stopped-forest laws are formalized
in Lean~4 at \url{https://github.com/YuZh98/aldous-broder-lean}.
The conditional-order laws, importance-correction results, and numerical
experiments are not formalized. Supplementary Section~\ref{sec:sm-lean}
gives the correspondence between the paper and the formal statements.

\section*{Acknowledgments}
We thank Dr. Leo Duan for helpful discussions about this project.

\section*{Declarations}
Generative-AI tools assisted with manuscript revision and coding.
The author assumes responsibility for all content.

\appendix
\section{Proofs of the auxiliary lemmas}\label{sec:app-proofs}

We give the deferred proofs of Lemmas~\ref{lem:adj},
\ref{lem:partition} and~\ref{lem:admissible}.

\begin{proof}[Proof of Lemma~\ref{lem:adj}]
(a) Since $A\one=\zero$, we have $\det A=0$, and thus $A\adj(A)=\zero$
by \eqref{eq:adj}. Therefore, every column of $\adj(A)$ lies in $\ker A$.
If $\operatorname{rank}A=|S|-1$, then $\ker A=\R\one$, and each column
is constant. If $\operatorname{rank}A\le |S|-2$, then every minor of
order $|S|-1$ vanishes, and $\adj(A)=\zero$. In both cases, the entry
$(v,u)$ of $\adj(A)$ equals its diagonal entry $d_u$, so
$\adj(A)=\one d^{\top}$, and $\adj(A)A=\zero$ gives
$d^{\top}A=\zero^{\top}$.

(b) We expand the determinant by linearity in each column, choosing
$A_{\cdot v}$ or $c_va$. The terms choosing $c_va$ in two columns
vanish. By \eqref{eq:adj} and part~(a), we have
\[
\det(A+ac^{\top})=\det A+\sum_v c_v\sum_u a_u\adj(A)_{vu}
=\sum_vc_v\sum_ua_ud_u=(c^{\top}\one)(a^{\top}d).
\]
This completes the proof.
\end{proof}

\begin{proof}[Proof of Lemma~\ref{lem:partition}]
We fix $v\in B^c$ and let $q_0\ge1$ be the first index with
$f^{(q_0)}(v)\in B$. We first prove the claim that the iterates
$f^{(q)}(v)$, $0\le q\le q_0$, are pairwise distinct. Suppose that
$f^{(q)}(v)=f^{(q')}(v)$ with $q<q'\le q_0$. Since an element of $B^c$
cannot equal one of $B$, both indices are below $q_0$. The iterates of
$f^{(q)}(v)$ are then periodic inside $B^c$ and never reach $B$, which
contradicts the definition of a $B$-parent map.

It follows that $v\in D_u$ exactly for $u=f^{(q)}(v)$, $0\le q\le q_0$.
Hence $v$ lies in the block of exactly one root, namely
$f^{(q_0)}(v)$, which proves the first statement. For the second, we
have $v\in D_u\setminus\{u\}$ if and only if $f^{(q)}(v)=u$ for some
$q\ge1$, that is, $v\in D_{w}$ for the child $w:=f^{(q-1)}(v)$ of $u$.
This child is unique by distinctness. Moreover, $u\notin D_{w}$ for a
child $w$ of $u$, since otherwise the iterates from $u$ would return to
$u$. This completes the proof.
\end{proof}

\begin{proof}[Proof of Lemma~\ref{lem:admissible}]
We first note that $(L\one)_u=1-\pi_u^{-1}(P^{\top}\pi)_u=0$. Since
$L[U]=\Pi[U]^{-1}\big((I-P)[U]\big)^{\top}\Pi[U]$, the two matrices have
the same determinant, and
\[
\big((I-P)[U]\big)^{-1}=\Pi[U]^{-1}\big(L[U]^{-1}\big)^{\top}\Pi[U],
\]
whose $(u,v)$ entry is $(\pi_v/\pi_u)(L[U]^{-1})_{vu}$.

Next, we prove \eqref{eq:green} and the positivity of the
determinant. Fix $u\in U$ and $w\in U^c$. By irreducibility, there is a path
of pairwise distinct states from $u$ to $w$ with positive transition
probabilities, and it leaves $U$ within at most $n-1$ steps. Hence the
walk started at $u$ leaves $U$ within $n-1$ steps with positive
probability, so every row of $P[U]^{n-1}$ sums to less than one, and the substochastic matrix $P[U]$ has spectral radius less than
$1$. Therefore, the Neumann series converges to
$\big((I-P)[U]\big)^{-1}$, which, together with the reversal, proves
\eqref{eq:green}. Moreover, the
eigenvalues $\eta$ of $P[U]$, lying in the open unit disk and closed
under conjugation, give
$\det\big((I-P)[U]\big)=\prod_\eta(1-\eta)>0$, the product over the
eigenvalues with multiplicity.

Finally, the first identity in \eqref{eq:ratio} is Cramer's rule. For
the second, the column $v$ of $L[U]^{-1}$, extended by zero off $U$ and
divided by its value at $v$, is $1$ at $v$, vanishes off $U$, and is
annihilated by the rows $U\setminus\{v\}$ of $L$, because
$L[U]\,(L[U]^{-1})_{\cdot v}=e_v$. By uniqueness, it equals
$\gamma^{U\setminus\{v\}}(e_v)$. This completes the proof.
\end{proof}

\clearpage
\section*{Supplementary materials}
\addcontentsline{toc}{section}{Supplementary materials}
\setcounter{section}{0}
\renewcommand{\thesection}{SM\arabic{section}}
\renewcommand{\theHsection}{SM\arabic{section}}
\setcounter{figure}{0}
\setcounter{table}{0}
\renewcommand{\thefigure}{SM\arabic{figure}}
\renewcommand{\thetable}{SM\arabic{table}}
\renewcommand{\theHfigure}{SM\arabic{figure}}
\renewcommand{\theHtable}{SM\arabic{table}}
These supplementary materials contain examples and secondary consequences
of the order-mass identity, details of the Lean formalization, and the
full numerical study. References to results without an SM prefix refer
to the main article above.

\section{Worked examples}\label{sec:app-examples}

We illustrate the results on three small examples. A four-state chain makes the harmonic
extensions, order masses and completion weights explicit. A chain with
one-way transitions shows the roles of reversal and support. A matrix
with entries of both signs checks the identity outside the
probabilistic setting, where individual order weights can be negative.

\begin{example}[A four-state chain]\label{ex:running}
Let $V=\{1,2,3,4\}$ and $L:=I-\tfrac15W^{\top}$ with
\[
W=\begin{pmatrix}1&1&1&2\\1&1&2&1\\1&2&1&1\\2&1&1&1\end{pmatrix}.
\]
Every column of $W$ sums to $5$, so $L\one=\zero$, and $L$ is the
matrix \eqref{eq:Ldef} of the doubly stochastic chain $P:=W/5$. We take
$B=\{1,2\}$, so that the minors required by Assumption~\ref{ass:hyp} are among
\[
\det L[\{1,2\}]=\tfrac35,\quad
\det L[\{1,2,3\}]=\det L[\{1,2,4\}]=\tfrac{36}{125}.
\]
The harmonic extension map from $B$ has the columns
\[
\tilde\Gamma^B=\begin{pmatrix}
\gamma^B(e_3) &\gamma^B(e_4)
\end{pmatrix},\quad
\gamma^{B}(e_3)=(\tfrac25,\tfrac35,1,0)^{\top},\quad
\gamma^{B}(e_4)=(\tfrac35,\tfrac25,0,1)^{\top}.
\]
We consider the three $B$-parent maps in Figure~\ref{fig:sm-running},
\[
f_1:\ 3\mapsto1,\ 4\mapsto3;\qquad
f_2:\ 3\mapsto2,\ 4\mapsto3;\qquad
f_3:\ 3\mapsto1,\ 4\mapsto2,
\]
with root blocks $(D_1,D_2)=(\{3,4\},\emptyset)$,
$(\emptyset,\{3,4\})$ and $(\{3\},\{4\})$, respectively.

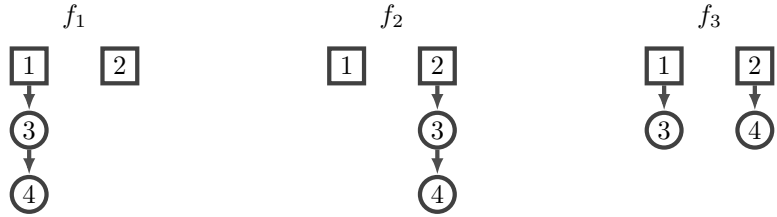
\begin{figure}[H]
\centering
\begin{tikzpicture}[x=1cm,y=1cm]
\foreach \x/\name in {0/f_1,4.2/f_2,8.4/f_3}{
  \node[pl] at (\x+0.6,0.95) {$\name$};}
\node[rootv] (a1) at (0,0.3) {$1$}; \node[rootv] (a2) at (1.2,0.3) {$2$};
\node[vtx] (a3) at (0,-0.55) {$3$}; \node[vtx] (a4) at (0,-1.4) {$4$};
\draw[ed] (a1)--(a3); \draw[ed] (a3)--(a4);
\node[rootv] (b1) at (4.2,0.3) {$1$}; \node[rootv] (b2) at (5.4,0.3) {$2$};
\node[vtx] (b3) at (5.4,-0.55) {$3$}; \node[vtx] (b4) at (5.4,-1.4) {$4$};
\draw[ed] (b2)--(b3); \draw[ed] (b3)--(b4);
\node[rootv] (c1) at (8.4,0.3) {$1$}; \node[rootv] (c2) at (9.6,0.3) {$2$};
\node[vtx] (c3) at (8.4,-0.55) {$3$}; \node[vtx] (c4) at (9.6,-0.55) {$4$};
\draw[ed] (c1)--(c3); \draw[ed] (c2)--(c4);
\end{tikzpicture}
\caption{The forests $f_1,f_2,f_3$ of Example~\ref{ex:running}, rooted at
$B=\{1,2\}$. Squares mark the roots, and edges point away from them.}
\label{fig:sm-running}
\end{figure}

On the right side of \eqref{eq:cofactor}, the matrices
$I-\Gamma^BM_f$ are
\[
\begin{pmatrix}0&0\\-1&1\end{pmatrix},\qquad
\begin{pmatrix}1&-1\\0&0\end{pmatrix},\qquad
\begin{pmatrix}\tfrac35&-\tfrac35\\[2pt]-\tfrac35&\tfrac35\end{pmatrix},
\]
with the vectors of principal cofactors $(1,0)^\top$, $(0,1)^\top$ and $(\tfrac35,\tfrac35)^\top$,
respectively. On the left side, $f_1$ has the single growth order
$(3,4)$, entering at $1$ with the factor $(\gamma^{B}(e_3))_3=1$, and
$f_2$ likewise enters at $2$. On the other hand, $f_3$ has $(3,4)$
entering at $1$ with the factor $(\gamma^{B}(e_3))_2=\tfrac35$, and
$(4,3)$ entering at $2$ with $(\gamma^{B}(e_4))_1=\tfrac35$. These agree
with the principal cofactors above.

For the single-root case, consider the single root $1$ and the star
$f_\star:2,3,4\mapsto1$. Every ordering of $\{2,3,4\}$ is a growth
order, and the six order weights
\[
\tfrac1{10},\ \tfrac3{20},\ \tfrac1{10},\ \tfrac3{20},\ \tfrac14,\ \tfrac14
\qquad\text{for }(2,3,4),(2,4,3),(3,2,4),(3,4,2),(4,2,3),(4,3,2)
\]
sum to $1$, as the single-root identity requires.
\end{example}

\begin{example}[Example~\ref{ex:running} continued]\label{ex:ab}
The principal cofactors of $I-P$ are $\lambda_v=\tfrac{36}{125}$ for
every $v$, so $\one^{\top}\lambda=\tfrac{144}{125}$. With the single
root $1$, the star $f_\star$ has tree weight
\[
w(F_\star)=\prod_{v\ne1}\pi_{f(v)}P_{f(v)v}=\prod_{v\ne1}\pi_1P_{1v}=\frac{1}{4000},
\]
and hence
\[
\pr(F=F_\star)=\frac{1/4000}{(144/125)(1/256)}=\frac1{18}.
\]
By Theorem~\ref{thm:orderlaw}, the six order weights of
Example~\ref{ex:running} are the conditional probabilities of the six
growth orders given $F=F_\star$. Moreover, the $16$ arborescences
rooted at $1$ have total weight $\tfrac9{2000}$, as the second form of
Theorem~\ref{thm:ab} requires.
\end{example}

\begin{example}[Example~\ref{ex:running} continued]\label{ex:stopped}
We return to the root set $B=\{1,2\}$ of Example~\ref{ex:running}, for
which $\pi$ is uniform and $\det L[B]=\tfrac35$. With the order masses
$(1,0)$, $(0,1)$ and $(\tfrac35,\tfrac35)$ of $f_1,f_2,f_3$,
\eqref{eq:phi} and \eqref{eq:closed} give the completion weights and
probabilities of three of the eight forests rooted at $B$ in
Table~\ref{tab:completion}.

\begin{table}[ht]
    \caption{Completion weights and laws of $F_B$ for the forests
    $f_1,f_2,f_3$ of Figure~\ref{fig:sm-running}. Here $\pr_r$ is the law of $F_B$
    under the seed $r$.}\label{tab:completion}
    \centering
    \small
    \begin{tabular}{@{}lccccc@{}}
        \toprule
        forest $F_0$ & $w(F_0)$ & $\varphi_1(F_0)$ & $\varphi_2(F_0)$ & $\pr_1(F_B=F_0)$ & $\pr_2(F_B=F_0)$ \\
        \midrule
        $f_1$ & $1/400$ & $1/5$ & $1/20$ & $1/9$ & $1/36$ \\
        $f_2$ & $1/200$ & $1/20$ & $1/5$ & $1/18$ & $2/9$ \\
        $f_3$ & $1/400$ & $3/20$ & $3/20$ & $1/12$ & $1/12$ \\
        \bottomrule
    \end{tabular}
\end{table}
\smallskip
The first two forests have weights in ratio $1{:}2$, but their seed-$1$
probabilities are in ratio $2{:}1$. Thus the walk from $1$ favors the
forest whose larger tree is rooted at the seed.
\end{example}

\begin{example}[A chain with one-way support]\label{ex:oneway}
The doubly stochastic chain of Example~\ref{ex:running} is convenient for
computation, but it hides the role of the reversal. To make this role
visible, we consider $V=\{1,2,3,4\}$ and the chain of
Figure~\ref{fig:sm-oneway},
\[
P=\begin{pmatrix}
0&\tfrac12&\tfrac12&0\\
0&0&\tfrac12&\tfrac12\\
\tfrac13&0&0&\tfrac23\\
1&0&0&0
\end{pmatrix},
\qquad
P^{*}=\begin{pmatrix}
0&0&\tfrac14&\tfrac34\\
1&0&0&0\\
\tfrac23&\tfrac13&0&0\\
0&\tfrac13&\tfrac23&0
\end{pmatrix},
\qquad
\]
with the stationary distribution
$\pi=(\tfrac13,\tfrac16,\tfrac14,\tfrac14)^{\top}$.

\begin{figure}[H]
\centering
\begin{tikzpicture}[x=1cm,y=1cm]
\node[vtx] (q1) at (0,0) {$1$};
\node[vtx] (q2) at (2.6,0) {$2$};
\node[vtx] (q3) at (0,-2.2) {$3$};
\node[vtx] (q4) at (2.6,-2.2) {$4$};
\draw[ed] (q1) -- node[lb,above]{$\tfrac12$} (q2);
\draw[ed] (q1) to[bend right=22] node[lb,left]{$\tfrac12$} (q3);
\draw[ed] (q3) to[bend right=22] node[lb,right]{$\tfrac13$} (q1);
\draw[ed] (q2) -- node[lb,right]{$\tfrac12$} (q4);
\draw[ed] (q3) -- node[lb,below]{$\tfrac23$} (q4);
\draw[ed] (q2) -- node[lb,pos=0.17,sloped,above]{$\tfrac12$} (q3);
\draw[ed] (q4) -- node[lb,pos=0.17,sloped,below]{$1$} (q1);
\end{tikzpicture}
\caption{The transitions of the chain of Example~\ref{ex:oneway}. Only the
arc between $1$ and $3$ is used in both directions, whereas every other arc is
one-way.}
\label{fig:sm-oneway}
\end{figure}
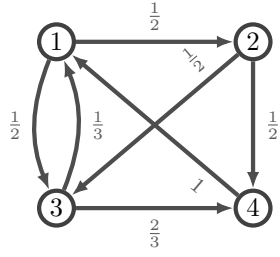

The chain is irreducible, since $1\to2\to4\to1$ and $1\to3\to1$ are
cycles through every state. However, nine of the sixteen transitions
vanish, the arcs $1\to2$, $2\to3$, $2\to4$, $3\to4$ and $4\to1$ are
one-way, and $\pi_uP_{uv}\ne\pi_vP_{vu}$ for every pair $u\ne v$. In
particular, $P$ is not positive on the edges of any undirected graph,
so the hypothesis of \cite{FredesMarckert2023} fails. Nevertheless,
the results of Section~\ref{sec:consequences} require only irreducibility,
and we now check each of them. All the numbers below are exact. For
comparison, we computed the laws of the walk independently by
first-step analysis on the pairs (visited set, current state), using
only the forward chain.

\emph{The tree law.} Start the covering walk at $r=1$. Of the $16$
arborescences rooted at $1$, only four have positive weight, and their
total weight is $\tfrac1{96}$. Table~\ref{tab:oneway-trees} lists these four
trees with their probabilities from Theorem~\ref{thm:ab} and the conditional
law of their growth orders from Theorem~\ref{thm:orderlaw}. The other twelve
trees have probability zero.

\begin{table}[ht]
\caption{The four trees of positive weight rooted at $1$ in
Example~\ref{ex:oneway}, written as $(f(2),f(3),f(4))$, with their weights,
their probabilities under the covering walk from $1$, and the order
weights of their growth orders.}\label{tab:oneway-trees}
\centering
\small
\begin{tabular}{@{}lccl@{}}
\toprule
tree & $w(F_0)$ & $\pr(F=F_0)$ & growth orders $\tau$ and order weights $\omega_f(\tau)$\\
\midrule
$(1,1,3)$ & $1/216$ & $4/9$ & $(2,3,4)\colon 0$,\quad $(3,2,4)\colon1/10$,\quad $(3,4,2)\colon9/10$\\
$(1,1,2)$ & $1/432$ & $2/9$ & $(2,3,4)\colon 0$,\quad $(2,4,3)\colon3/4$,\quad $(3,2,4)\colon1/4$\\
$(1,2,3)$ & $1/432$ & $2/9$ & $(2,3,4)\colon 1$\\
$(1,2,2)$ & $1/864$ & $1/9$ & $(2,3,4)\colon1/4$,\quad $(2,4,3)\colon3/4$\\
\bottomrule
\end{tabular}
\end{table}

The probabilities and the order weights both agree with the first-step
analysis. Every tree of positive weight has $f(2)=1$, since the only arc into
$2$ is from $1$. In contrast, the law proportional to
$\prod_{v\ne1}P_{vf(v)}$, which Wilson's algorithm samples when run
with $P$ in place of $P^{*}$, is supported on the
trees with $f(2)\in\{3,4\}$, since $2$ can only move to $3$ or $4$.
The two laws are therefore mutually singular.

\emph{Order weights equal to zero.} The growth order $(2,3,4)$ of the
first two trees has weight zero. Its first factor is
$(\gamma^{\{1\}}(e_2))_1=P^{*}_{12}=0$. Correspondingly, once the walk
has entered $2$, it cannot return to $1$ without visiting $3$ or $4$,
so it never discovers $2$ and then $3$ from $1$. The order weights of
each of the $16$ trees still sum to one, as Theorem~\ref{thm:cofactor}
requires, including the twelve trees of weight zero.

\emph{The stopped covering walk.} Take $B=\{1,2\}$. Then
$L[B]=\big(\begin{smallmatrix}1&0\\-1&1\end{smallmatrix}\big)$, so
$\det L[B]=1$, and \eqref{eq:phi} gives
$\varphi_1(F_0)=\pi_2\big(g_1(B,f)+g_2(B,f)\big)$ and
$\varphi_2(F_0)=\pi_1\,g_2(B,f)$. Table~\ref{tab:oneway-forests} lists the
eight forests rooted at $B$ with their order masses, completion weights
and probabilities under the two seeds. Both probability columns agree
with \eqref{eq:closed} and with the first-step analysis.

\begin{table}[ht]
\caption{The eight forests rooted at $B=\{1,2\}$ in Example~\ref{ex:oneway},
written as $(f(3),f(4))$, with their weights, order masses, completion
weights at the seed $2$, and laws $\pr_r$ of $F_B$ under the seeds
$r=1,2$. The completion weight at the seed $1$ equals $\tfrac16$ for
every forest.}\label{tab:oneway-forests}
\centering
\small
\begin{tabular}{@{}lccccc@{}}
\toprule
forest $F_0$ & $w(F_0)$ & $g(B,f)$ & $\varphi_2(F_0)$ & $\pr_1(F_B=F_0)$ & $\pr_2(F_B=F_0)$\\
\midrule
$(1,3)$ & $1/36$ & $(1,0)$ & $0$ & $4/9$ & $0$\\
$(1,2)$ & $1/72$ & $(1/4,\,3/4)$ & $1/4$ & $2/9$ & $1/3$\\
$(2,3)$ & $1/72$ & $(0,1)$ & $1/3$ & $2/9$ & $4/9$\\
$(2,2)$ & $1/144$ & $(0,1)$ & $1/3$ & $1/9$ & $2/9$\\
$(1,1)$ & $0$ & $(1,0)$ & $0$ & $0$ & $0$\\
$(2,1)$ & $0$ & $(3/4,\,1/4)$ & $1/12$ & $0$ & $0$\\
$(4,1)$ & $0$ & $(1,0)$ & $0$ & $0$ & $0$\\
$(4,2)$ & $0$ & $(0,1)$ & $1/3$ & $0$ & $0$\\
\bottomrule
\end{tabular}
\end{table}

The last four forests have weight zero, since $P_{14}=P_{43}=0$, but their order
masses are well defined and satisfy Theorem~\ref{thm:cofactor}. Indeed,
the identity concerns $L$, not the weight of the forest.

The two seeds behave very differently. Every order mass vector sums to
one here, so $\varphi_1(F_0)=\pi_2=\tfrac16$ for every $F_0$, and hence
$\pr_1(F_B=F_0)\propto w(F_0)$. Thus the walk from $1$ samples exactly
the weight-proportional law. This is an instance of
Remark~\ref{rem:unbiased}, since every completion at $1$ whose added edge has
positive weight adds the edge $1\to2$, the only arc into $2$, so that
$\varphi_1(F_0)=\pi_1P_{12}=\tfrac16$ for every $F_0$.

In contrast, the walk from the seed $2$ is tilted by $\varphi_2$ and
never returns $(1,3)$, the forest of largest weight. Indeed, its first
step enters $3$ or $4$ from $2$, and the completion weight $\varphi_2$
of $(1,3)$ vanishes, since no vertex of the tree of $2$ has an arc into
$1$. Since $(1,3)$ has probability $\tfrac49$ under the
weight-proportional law, the target mass lost from the seed $2$ is
$\tfrac49$. By Remark~\ref{rem:importance}, every reweighting of
$\pr_2$ is therefore at total variation distance at least $\tfrac49$
from that law.

Finally, \eqref{eq:phi} also holds for the forests of weight zero, as
the continuity argument in the proof of Theorem~\ref{thm:stopped}
shows. For example,
the forest $(2,1)$ has a single completion rooted at $2$ whose added
edge has positive weight, the edge $3\to1$, so that
$\varphi_2(F_0)=\pi_3P_{31}=\tfrac1{12}=\pi_1\,g_2(B,f)$, although the
completed tree has weight zero.
\end{example}

\begin{example}[A matrix with entries of both signs]\label{ex:signed}
We now consider a matrix $L$ that no Markov chain produces, so that
some entries of the harmonic extensions are negative and some exceed
one. Let
$V=\{1,\dots,5\}$ and
\[
L=\begin{pmatrix}
3&-1&1&-2&-1\\
-1&3&-2&1&-1\\
1&-2&3&-1&-1\\
-1&-1&-1&3&0\\
-1&1&-1&-1&2
\end{pmatrix},
\]
whose rows sum to zero and whose off-diagonal entries
$L_{13},L_{24},L_{31},L_{52}$ are positive. We take $B=\{1,2\}$ and
$f:3\mapsto1,\ 4\mapsto3,\ 5\mapsto2$, with root blocks $D_1=\{3,4\}$
and $D_2=\{5\}$, which is the forest of Figure~\ref{fig:example}. The three growth orders $(3,4,5)$, $(3,5,4)$, $(5,3,4)$
have the proper bases $\{1,2\}$, $\{1,2,3\}$, $\{1,2,5\}$, $\{1,2,3,4\}$,
$\{1,2,3,5\}$, on which the minors of $L$ are $8$, $13$, $16$, $31$
and $23$, respectively. The rows of $\tilde\Gamma^B$ indexed by $B$ are
\[
\Gamma^{B}=\begin{pmatrix}-\tfrac18&\tfrac58&\tfrac12\\[2pt]
\tfrac58&-\tfrac18&\tfrac12\end{pmatrix},
\]
and the aggregated matrix is
\[I-\Gamma^BM_f
=\left(\begin{smallmatrix}\frac12&-\frac12\\-\frac12&\frac12\end{smallmatrix}\right),\]
with both principal cofactors equal to $\tfrac12$. In the base
$\{1,2,3\}$,
$\Gamma^{\{1,2,3\}}_{\cdot4}=(\tfrac8{13},-\tfrac1{13},\tfrac1{13})^{\top}$
and
$\Gamma^{\{1,2,3\}}_{\cdot5}=(\tfrac5{13},\tfrac{14}{13},\tfrac{12}{13})^{\top}$. The order weights are
\[
\omega_f(3,4,5)= 1\cdot\big(-\tfrac1{13}\big)=-\tfrac1{13},\qquad
\omega_f(3,5,4)= \tfrac58\cdot\tfrac{12}{13}=\tfrac{15}{26},\qquad
\omega_f(5,3,4)= \tfrac12\cdot1=\tfrac12 .
\]
The first two orders enter at the root $1$ and the third at $2$, so the
order masses are $-\tfrac1{13}+\tfrac{15}{26}=\tfrac12$ at $r=1$ and
$\tfrac12$ at $r=2$, as Theorem~\ref{thm:cofactor} asserts. In particular, the
check at $r=1$ balances a negative weight against a weight whose
factors come from two different bases.
\end{example}

\FloatBarrier
\section{Forest sums and linear extensions}\label{sec:app-comb}

\begin{remark}[Forest sums]\label{rem:forest}
In general, by the all-minors matrix-tree theorem
\cite{Chaiken1982}, $\det L[U]$ is a signed sum over the forests rooted
at $U^c$. Hence Assumption~\ref{ass:hyp} asks that these forest sums do not
vanish.

By Cramer's rule and the all-minors matrix-tree theorem \cite{Chaiken1982},
each entry $\Gamma^U_{uv}$, $u\in U$, $v\in U^c$, is the signed sum over
the forests rooted at $U^c$ in which $u$ lies in the tree of $v$,
divided by the signed sum over all forests rooted at $U^c$, each forest
counted by the product of the off-diagonal entries of $-L$ along its
arcs. In the chain instance, the signs are positive, and this is the
harmonic tree formula \cite{PitmanTang2018}, \cite{Pokarowski1999}.
On the right, the matrix-tree theorem \cite{Tutte1948,Chaiken1982} gives
\begin{equation}\label{eq:arborescence}
\det\big((I-\Gamma^BM_f)[B\setminus\{r\}]\big)
=\sum_{T}\ \prod_{u\in B\setminus\{r\}}\big(\Gamma^BM_f\big)_{u,\,T(u)},
\end{equation}
where the sum is over arborescences $T$ on $B$ rooted at $r$, with
$T(u)$ the parent of $u$. The order-mass identity therefore equates
the total weight of growth orders whose first vertex is a child of
$r$ with the total weight of these arborescences. No weight-preserving
correspondence between the two index sets is known: the sum defining
$g_r(B,f)$ has $|\{\tau\in\Orders_B(f):f(\tau_1)=r\}|$ terms,
whereas the expansion \eqref{eq:arborescence} has at most $|B|^{|B|-2}$.
\end{remark}
We now specialize the identity to harmonic extensions expressed
through positive vertex weights. Recall that successive sampling
without replacement from $B^c$ with positive weights $p$ draws the
ordering $\tau$ with probability
$\prod_{\ell=1}^{k}p_{\tau_\ell}/p(\{\tau_\ell,\dots,\tau_k\})$, which is
the Plackett-Luce law \cite{Luce1959,Plackett1975}. The following
corollary expresses the order-mass identity directly in these weights.

\begin{corollary}[Successive sampling on forests]\label{cor:PL}
Let $f$ be a $B$-parent map with root blocks $D_u$, let $p_v>0$ for
$v\in B^c$, write $p(J):=\sum_{v\in J}p_v$ for $J\subseteq B^c$, and for a growth order
$\tau$ let $B_\ell$ be its bases. Then, for every $r\in B$,
\begin{equation}\label{eq:PL}
\sum_{\tau\in\Orders_B(f):\ f(\tau_1)=r}\
\prod_{1\le \ell\le k-1:\ f(\tau_{\ell+1})\ne\tau_\ell}\frac{p_{\tau_\ell}}{p(B_{\ell-1}^c)}
=\frac{p(D_r)}{p(B^c)} .
\end{equation}
In particular, with $B=\{r\}$ and $p$ uniform on $B^c$,
\begin{equation}\label{eq:Kn}
\sum_{\tau\in\Orders_{\{r\}}(f)}\
\prod_{1\le \ell\le k-1:\ \tau_\ell=f(\tau_{\ell+1})}(k-\ell+1)=k!\,.
\end{equation}
\end{corollary}

\begin{proof}[Proof of Corollary~\ref{cor:PL}]
We extend $p$ to $V$ by positive values on $B$ and normalize its sum to
one, which leaves both sides of \eqref{eq:PL} unchanged. We set
$L:=I-\one p^{\top}$. Then $L$ has zero row sums, $\det L[U]=1-p(U)>0$
for every proper $U$ by \eqref{eq:adj}, and $L[U]\one=(1-p(U))\one$.
Hence, for $v\in U^c$,
\[
\Gamma^U_{\cdot v}=p_vL[U]^{-1}\one=\frac{p_v}{p(U^c)}\one .
\]
Therefore, the factor at $\ell$ is $p_{\tau_\ell}/p(B_{\ell-1}^c)$
when $f(\tau_{\ell+1})\ne\tau_\ell$, and $1$ otherwise. On the right, with
$q_u:=p(D_u)/p(B^c)$ for $u\in B$, we have $\Gamma^BM_f=\one q^{\top}$ and,
by \eqref{eq:adj},
\[
\det\big((I-\one q^{\top})[B\setminus\{r\}]\big)=1-\sum_{u\ne r}q_u=q_r ,
\]
so Theorem~\ref{thm:cofactor} gives \eqref{eq:PL}.
For uniform $p$ and $B=\{r\}$, we multiply each summand by
$k!=\prod_{\ell=1}^{k-1}(k-\ell+1)$ to obtain \eqref{eq:Kn}.
\end{proof}

\begin{remark}[Linear extensions]\label{rem:linext}
In \eqref{eq:Kn}, the weight of a linear extension depends on the
positions at which a vertex is immediately followed by its child.
For example, a star has
$k!$ linear extensions of weight one, and a path has one, of weight
$k!$.
\end{remark}

\section{Completion weights}\label{sec:sm-completion}

\begin{remark}[The contracted chain]\label{rem:contract}
The sum of tree weights over completions is a weighted analogue of
the number of spanning trees containing a fixed forest. Such tree
marginals can be computed by contraction \cite{LyonsPeres2016},
which also appears in the two-stage samplers of
\cite{NunesIacobelliFigueiredo2025}. In our setting, the matrix
\[
Q:=P^{*}[B]+P^{*}[B,B^c]M_f
\]
is a transition matrix on $B$. After one step of the reversed chain from a root, we record the root of the tree containing
the destination. This construction gives a second evaluation of the
completion weight.

To derive this evaluation, consider an added edge into a root
$u\ne r$. If its parent
$x$ lies in the tree of $u'$, the sum of its weights over that tree is
$\sum_x\pi_xP_{xu}=\pi_uQ_{uu'}$. Group the completions by the tree
obtained after contracting each component of $F_0$ to its root.
For each contracted tree, summing over the possible endpoints of its
added edges gives the corresponding product of these flows. The
matrix-tree theorem \cite{Tutte1948,Chaiken1982}, applied to
$I-Q=L[B]+L[B,B^c]M_f$, therefore gives \eqref{eq:phiN}.
In particular,
$\varphi_r(F_0)/\prod_{u\in B\setminus\{r\}}\pi_u$ is the principal
cofactor at $r$ of $I-Q$. These cofactors form an unnormalized
invariant measure of $Q$ by Lemma~\ref{lem:adj}(a). For a fixed seed, larger
cofactors correspond to larger completion weights.

To relate this construction to the order masses, compare $Q$ with
$\Gamma^BM_f$ from Remark~\ref{rem:reading}. The latter records the
tree entered when the reversed chain first exits $B$, after any
number of steps inside $B$, whereas $Q$ records one step followed
by aggregation. The two matrices satisfy
$I-Q=L[B]\,(I-\Gamma^BM_f)$. The product rule for adjugates then gives
the cofactor row of $I-Q$ as
$\det L[B]\,g(B,f)^{\top}L[B]^{-1}$, in agreement with
\eqref{eq:phi} and \eqref{eq:phiN}. Thus the order masses are the
cofactors of $I-\Gamma^BM_f$, whereas the completion weights are the
cofactors of $I-Q$ multiplied, at each root $r$, by
$\prod_{u\in B\setminus\{r\}}\pi_u$.

\end{remark}

\begin{remark}[Computing the completion weight]\label{rem:cost}
For a fixed chain and root set, form $L[B]$ and $L[B,B^c]$ once.
Writing $b:=|B|$ and $k:=|B^c|$, each forest then requires $O(bk)$
operations to assemble $L[B]+L[B,B^c]M_f$ and $O(b^3)$ to evaluate
its principal minor. Forests with the same root blocks share this
calculation.
\end{remark}

\section{The identity for real aggregations}\label{sec:sm-algebra}

The order-mass identity depends on the forest only through its root
indicator. We now extend it to any real aggregation with unit row sums,
remove the root set by doubling the vertex set, and quantify the defect
when one aggregation row does not sum to one.

In this section, for $U\subsetneq V$ with $\det L[U]\ne0$,
\[
\Gamma^U=-L[U]^{-1}L[U,U^c]\in\R^{U\times U^c},\qquad
\tilde\Gamma^U=\begin{pmatrix}\Gamma^U\\ I_{|U^c|}\end{pmatrix}\in\R^{V\times U^c}.
\]

\begin{assumption}[All bases]\label{ass:all}
$L\one=\zero$, and $\det L[U]\ne0$ for every $U$ with
$B\subseteq U\subsetneq V$.
\end{assumption}

\begin{corollary}[Real aggregation]\label{cor:soft}
Let $B\subsetneq V$ be nonempty, $k:=|B^c|$, let $L$ satisfy
Assumption~\ref{ass:all}, and let $M\in\R^{B^c\times B}$ satisfy $M\one=\one$.
Then, for every $r\in B$,
\begin{equation}\label{eq:soft}
\sum_{\tau}\ M_{\tau_1 r}\prod_{\ell=1}^{k-1}\
\sum_{u\in B}M_{\tau_{\ell+1}u}\,\Gamma^{B_{\ell-1}}_{u\tau_\ell}
=\det\big((I-\Gamma^BM)[B\setminus\{r\}]\big),
\end{equation}
the outer sum over all $k!$ orderings $\tau$ of $B^c$, with
$B_\ell:=B\cup\{\tau_1,\dots,\tau_\ell\}$.
\end{corollary}

\begin{proof}
We first consider the case in which every row of $M$ is a standard
basis vector, $M_{v\cdot}=e_{f(v)}^{\top}$ for a map $f:B^c\to B$. Then $M$ is the root
indicator of the $B$-parent map $f$, whose growth orders are all
orderings of $B^c$ and whose bases are all sets $B\subseteq U\subsetneq V$.
The left side of \eqref{eq:soft} is $g_r(B,f)$, since
$M_{\tau_1r}=1(f(\tau_1)=r)$
and the inner sum is
$\tilde\Gamma^{B_{\ell-1}}_{f(\tau_{\ell+1})\tau_\ell}$,
so \eqref{eq:soft} is Theorem~\ref{thm:cofactor}.

For general $M$, we induct on the number of rows that are not standard
basis vectors. We first show that both sides of \eqref{eq:soft} are affine
functions of any one row $m:=M_{v\cdot}$ when the other rows are fixed.
On the left, each term contains the row $v$ exactly once, as
$M_{\tau_1r}$ if $\tau_1=v$, and in the inner sum at the $\ell$ with
$\tau_{\ell+1}=v$ otherwise. On the right, we have
$I-\Gamma^BM=N_0-\Gamma^B_{\cdot v}\,m^{\top}$ with $N_0$ independent of
$m$, and a determinant whose columns are affine in $m$ along one fixed
vector is affine in $m$, since the terms with two such columns vanish.
Now suppose that the row $v$ is not a standard basis vector. Then the
difference of the two sides, as a function of $m$ on the hyperplane
$\{\one^{\top}m=1\}$, is affine and vanishes at the $e_u$, $u\in B$, by
the induction hypothesis. Since $m=\sum_um_ue_u$ with $\sum_um_u=1$, the
difference at $m$, by affinity, is the weighted sum of its values at
$e_u$, with weights $m_u$, and is therefore zero. This completes the proof.
\end{proof}

The aggregation identity also yields a formula without a prescribed
root set.

\begin{proposition}[The identity without a root set]\label{prop:reduced}
Let $L\in\R^{V\times V}$ satisfy $L\one=\zero$ and $\det L[U]\ne0$ for
every $U$ with $\emptyset\ne U\subsetneq V$, and let $M\in\R^{V\times V}$
have unit row sums. Then, for every $u\in V$,
\begin{equation}\label{eq:reduced}
\sum_{\tau}\ M_{\tau_1 u}\prod_{\ell=1}^{n-1}\
\sum_{v\in V}M_{\tau_{\ell+1}v}\,\tilde\Gamma^{U_{\ell-1}}_{v\tau_\ell}
=\det\big((I-M)[V\setminus\{u\}]\big),
\end{equation}
the sum over all $n!$ orderings $\tau$ of $V$, with
$U_\ell:=\{\tau_1,\dots,\tau_\ell\}$. The factor $\ell=1$ is $M_{\tau_2\tau_1}$,
since $\tilde\Gamma^{\emptyset}=I$.
\end{proposition}

\begin{proof}
The key step is to double the vertex set. Let $B:=\{1',\dots,n'\}$ be a
copy of $V$ and $V':=B\sqcup V$. In the block order $(B,V)$, we let
$L':=\left(\begin{smallmatrix}I_n&-I_n\\ \zero&L\end{smallmatrix}\right)$,
and read $M$ as an aggregation of $V$ into $B$. Then Assumption~\ref{ass:all}
holds for $(L',B)$. Indeed, $L'\one=\zero$, block triangularity gives
$\det L'[B\cup U]=\det L[U]\ne0$ for nonempty proper $U\subset V$, and
$L'[B]=I_n$ covers $U=\emptyset$. It remains to identify the two sides
of \eqref{eq:soft} for $(L',B,M)$.

On the right side of \eqref{eq:soft},
the matrix $\Gamma^B$ of $B$ for $L'$ is $-L'[B]^{-1}L'[B,V]=I$, so the
aggregated matrix is $I-M$, whose cofactor at $u'$ is the right side of
\eqref{eq:reduced}.

On the left side of \eqref{eq:soft}, for a proper $U\subset V$
and $w\in U^c$, the harmonic extension $x$ of $e_w$ into $B\cup U$ for
$L'$ is defined by $x|_{V\setminus U}=e_w|_{V\setminus U}$,
$x_{v'}=x_v$ for $v\in V$, and $L[U,V]\,x|_V=\zero$. Uniqueness gives
$x|_V=\gamma^{U}(e_w)$ by Lemma~\ref{lem:harmonic-extension}, which for
$U=\emptyset$ reads $x|_V=e_w$, and $x|_B$ is its copy. For $U=U_{\ell-1}$ and
$w=\tau_\ell$ each inner factor of \eqref{eq:soft} becomes
$\sum_{v\in V}M_{\tau_{\ell+1}v}\tilde\Gamma^{U_{\ell-1}}_{v\tau_\ell}$, and the
initial factor is $M_{\tau_1u'}=M_{\tau_1u}$. Therefore, Corollary~\ref{cor:soft}
for $(L',B,M)$ at $u'$ is exactly \eqref{eq:reduced}. This completes
the proof.
\end{proof}

In this construction, $\Gamma^B=I$, so the aggregated matrix is $I-M$
for every $L$. The matrix $L$ enters only through the harmonic
extensions inside $V$ on the left side of \eqref{eq:reduced}.

The affinity in the proof of Corollary~\ref{cor:soft} also gives the
exact effect of a row with another sum.

\begin{proposition}[The defect of a non-unit row]\label{prop:rowsum}
Let $L$ satisfy Assumption~\ref{ass:all}, fix $v\in B^c$ and $r\in B$, and let
$M\in\R^{B^c\times B}$ have unit row sums in every row other than the row
$v$. As functions of $m:=M_{v\cdot}$, the two sides of \eqref{eq:soft}
satisfy
\[
\mathrm{LHS}(m)-\mathrm{RHS}(m)=\big(\one^{\top}m-1\big)\cdot
\det\big((I-\Gamma^BM^{0})[B\setminus\{r\}]\big),
\]
where $M^{0}$ is $M$ with its row $v$ replaced by $\zero$.
\end{proposition}

\begin{proof}
The difference $h:=\mathrm{LHS}-\mathrm{RHS}$ is affine in $m$, say
$h(m)=a^{\top}m+a_0$, and vanishes on $\{\one^{\top}m=1\}$ by
Corollary~\ref{cor:soft}. For $m_0$ on that hyperplane and $z$ with
$\one^{\top}z=0$, we have $a^{\top}z=h(m_0+z)-h(m_0)=0$, so $a=c\one$
for a scalar $c$, and $h(m_0)=0$ gives $a_0=-c$. At $m=\zero$, every
term of the left side contains a factor linear in $m$, so
$\mathrm{LHS}(\zero)=0$ and $c=-h(\zero)=\mathrm{RHS}(\zero)$, which is
the displayed cofactor. This completes the proof.
\end{proof}

Hence, for fixed other rows and a fixed root $r$, \eqref{eq:soft} holds
for a row with $\one^{\top}m\ne1$ if and only if the cofactor
$\det\big((I-\Gamma^BM^{0})[B\setminus\{r\}]\big)=0$. Unit row
sums are therefore sufficient for every configuration and every root,
and they cannot be dropped in general, but they are not necessary
for a fixed configuration and root. For example, let $V=\{1,2,3,4\}$,
$B=\{1,2\}$, $L=I-\one\one^{\top}/4$, $M_{3\cdot}=(a,b)$ and
$M_{4\cdot}=(-1,2)$. Then every entry of $\Gamma^B$ equals $\tfrac12$,
the cofactor vanishes at $r=1$, and both sides of \eqref{eq:soft} equal
$-\tfrac b2$ at $r=1$ for all $a,b$. At $r=2$, the cofactor is $\tfrac32$,
and the difference of the two sides is $\tfrac32(a+b-1)$.

The zero row sums of $L$ cannot be dropped either. For example,
for $|B|=2$, $|B^c|=1$ and $f(w)=1$, the cofactor at the root $2$ is
$1-\Gamma^B_{1w}$, which vanishes by \eqref{eq:unitsum}, as the left side
of \eqref{eq:cofactor} at $r=2$ does. Without zero row sums in the
rows $B$, however, the entry $\Gamma^B_{1w}=-(L[B]^{-1}L[B,\{w\}])_1$
can take any value, so the cofactor need not vanish.

\section{Randomizing the seed does not remove the tilt}\label{sec:sm-seed}

By Theorem~\ref{thm:stopped}, the stopped covering walk samples the
weight-proportional law tilted by the completion weight. A natural
question is whether a mixture over seeds removes the tilt. We now show
that it does not in general.

\begin{proposition}[No seed mixture is unbiased in general]\label{prop:notilt}
Let $P$ be irreducible and stochastic with stationary distribution
$\pi$, let $\lambda$ be as in \eqref{eq:lam}, let $B\subsetneq V$ be
nonempty, and for $r\in B$ let $\pr_r$ be the
law of the first-entrance forest of the stopped covering walk from $r$.
For every $c\in\R^{B}$ and every $F_0\in\Fset_B$ with parent map $f$,
\begin{equation}\label{eq:mixture}
\sum_{r\in B}c_r\,\pr_r(F_B=F_0)
=\frac{w(F_0)}{(\one^{\top}\lambda)\prod_{v\in V}\pi_v}
\sum_{r\in B}c_r\,\varphi_r(F_0).
\end{equation}
For the chain of Example~\ref{ex:running} with $B=\{1,2\}$, no $c\ne\zero$
makes the left side proportional to the forest weight on $\Fset_B$.
\end{proposition}

\begin{proof}
Formula \eqref{eq:mixture} follows immediately from \eqref{eq:closed}. In the example, every forest
has positive weight, so proportionality to the forest weight forces
$\sum_rc_r\varphi_r(F_0)$ to take the same value for every $F_0\in\mathcal F_B$. According to Table~\ref{tab:completion}, $(\varphi_1,\varphi_2)$ equals
$(\tfrac15,\tfrac1{20})$, $(\tfrac1{20},\tfrac15)$ and
$(\tfrac3{20},\tfrac3{20})$ at $f_1$, $f_2$ and $f_3$, respectively. Equating the values at $f_1$ and
$f_2$ gives $c_1=c_2=:c$. The values at $f_1$ and $f_3$ are then
$(\tfrac15+\tfrac1{20})c=\tfrac c4$ and $(\tfrac3{20}+\tfrac3{20})c=\tfrac{3c}{10}$, forcing $c=0$. This completes the proof.
\end{proof}

To see where the tilt comes from, let $B=\{r,u\}$. A completion of $F_0$ at $r$ adds one edge into $u$ from
the tree of $r$, so
$\varphi_r(F_0)=\sum_{x\in D_r\cup\{r\}}\pi_xP_{xu}$ is the stationary
flow from the tree of $r$ into $u$. Hence the walk from $r$ favors the
forests whose tree of $r$ sends much stationary flow into $u$, and
randomizing the seed only averages two such preferences.

\section{Lean formalization}\label{sec:sm-lean}

\begin{sloppypar}
We formalize the order-mass identity, the Aldous-Broder theorem,
and the stopped-forest law in Lean~4
\cite{deMouraUllrich2021} with mathlib \cite{mathlib2020} (toolchain
v4.30.0, mathlib v4.30.0), in the directory \texttt{EAB/Paper}
of the repository \url{https://github.com/YuZh98/aldous-broder-lean}. The development imports mathlib only,
follows the proofs of this article, and contains no \texttt{sorry}. The four endpoint theorems \texttt{order\_mass\_identity},
\texttt{Chain.aldous\_broder}, \texttt{Chain.aldous\_broder\_normalized}
and \texttt{Chain.stopped\_forest\_law}, together with their principal
proof obligations, depend only on the axioms \texttt{propext},
\texttt{Classical.choice} and \texttt{Quot.sound}, which
\texttt{lake env lean AxiomCheck.lean} prints after \texttt{lake build}.
Table~\ref{tab:lean} lists the correspondence. All declarations are in
the namespace \texttt{EAB.Paper}.
\end{sloppypar}

\begin{table}[ht]
\caption{Results and their Lean declarations.}\label{tab:lean}
\centering
\footnotesize
\begin{tabular}{@{}>{\raggedright\arraybackslash}p{3.3cm}>{\raggedright\arraybackslash}p{8.4cm}@{}}
\toprule
result & Lean declaration \\
\midrule
Definition~\ref{def:gamma}, \eqref{eq:unitsum} & \texttt{Foundation.harmonic}, \texttt{Foundation.gamma}, \texttt{Foundation.harmonic\_unique}, \texttt{Foundation.harmonic\_one} \\
\eqref{eq:Gamma} & $\Gamma^U_{uv}$ is \texttt{Foundation.gamma L U v u} for $u\in U$ \\
Lemma~\ref{lem:adj} & \texttt{Foundation.adjugate\_apply\_eq\_principalCofactor}, \texttt{Foundation.principalCofactor\_left\_kernel}, \texttt{Foundation.det\_add\_rank\_one\_zero\_row\_sum} \\
Definition~\ref{def:forest}, Lemma~\ref{lem:partition} & \texttt{IsPaperParentMap}, \texttt{growthOrders}, \texttt{rootIndicator}, \texttt{rootBlock\_pairwiseDisjoint}, \texttt{rootBlocks\_cover}, \texttt{childDescendants\_cover} \\
\eqref{eq:weight}, Definition~\ref{def:ordermass} & \texttt{orderWeight}, \texttt{orderMass} \\
Lemma~\ref{lem:transfer} & \texttt{Foundation.one\_vertex\_transfer}, \texttt{Foundation.one\_vertex\_transfer\_singleton} \\
Theorem~\ref{thm:cofactor} & \texttt{order\_mass\_identity}, \texttt{single\_root\_normalization} \\
Lemma~\ref{lem:admissible} & \texttt{Chain.det\_principal\_reversedMatrix\_pos}, \texttt{Chain.green\_hasSum}, \texttt{Chain.green\_reversal}, \texttt{Chain.inv\_diag\_eq\_det\_ratio}, \texttt{Chain.inv\_column\_eq\_gamma} \\
Lemma~\ref{lem:stationary} & \texttt{Chain.isStationary\_cofactor}, \texttt{Chain.IsStationary.cofactorVec\_eq} \\
Proposition~\ref{prop:orderlaw} & \texttt{Chain.chainLaw\_orderEvent\_eq} \\
Theorem~\ref{thm:ab} & \texttt{Chain.aldous\_broder}, \texttt{Chain.aldous\_broder\_normalized} \\
Theorem~\ref{thm:stopped} & \texttt{Chain.stopped\_forest\_law} \\
\bottomrule
\end{tabular}
\end{table}

\begin{sloppypar}
The Lean statements differ from the text in the following modeling
choices. The law of the chain started at $r$ is mathlib's Ionescu-Tulcea
trajectory measure (\texttt{Chain.chainLaw}), checked against the
cylinder formula (\texttt{Chain.chainLaw\_cylinder}), and probabilities
are extended nonnegative reals. The forest event requires the walk to
visit every state, which has probability one
(\texttt{Chain.ae\_hits\_all}), and its first-entrance parents to equal
$f$. This is the event $\{F_B=F_0\}$ up to a null set. The stationary
distribution enters as a hypothesis that $\pi$ is positive, normalized
and invariant, and \texttt{Chain.IsStationary.cofactorVec\_eq} proves
that such a $\pi$ is the normalized cofactor vector, hence unique. The
standing assumption $n\ge2$ appears as \texttt{2 \(\le\) Fintype.card V}.
The formal version of Theorem~\ref{thm:cofactor} is stated over an
arbitrary field, and
\texttt{Chain.stopped\_forest\_law} holds for every parent assignment,
including those of weight zero. It states \eqref{eq:phi} together with
the law obtained by substituting \eqref{eq:phi} into
\eqref{eq:closed}. Lemma~\ref{lem:partition} is formalized for
the children of roots, which is the case the proof uses, and
\texttt{Foundation.one\_vertex\_transfer} allows a fold along any
weights that sum to one, of which Lemma~\ref{lem:transfer} is the fold into
a single root.
\end{sloppypar}

Proposition~\ref{prop:orderlaw} is formalized for
growth orders in its general form \eqref{eq:orderlaw}. Its statement for
the other orderings and the simplification \eqref{eq:orderlaw1} are not
stated separately. We have not formalized Theorem~\ref{thm:orderlaw}, Corollary~\ref{cor:stoppedorder},
Corollary~\ref{cor:PL}, the results of
Sections~\ref{sec:sm-algebra} and~\ref{sec:sm-seed}, and the numerical
experiments.

\section{Numerical illustration}\label{sec:numerics}

We use small finite examples to examine three consequences of the
stopped-forest law: its distance from the weight-proportional law,
the target mass lost through incomplete support, and the variability
of importance-corrected estimates. We generate $30$ chains in each
of three families on $V=\{1,\dots,7\}$. Complete enumeration of the
forest space gives both laws without sampling. The experiments
illustrate the formulas and check their implementation. Estimating
the typical bias on larger graphs is beyond their scope.

\subsection{Setting}\label{sec:sm-numerics}
Table~\ref{tab:families} specifies the dense, graph, and drift families.
For every chain, we use the nested root sets
$B=\{1\},\{1,2\},\dots,\{1,\dots,6\}$ with seed $1$.

In the drift family, changing the holding weights leaves the jump
chain unchanged and multiplies every off-diagonal stationary flow
$\pi_uP_{uv}$ by a common constant. The $30$ drift chains therefore
share the same forest laws and give one example for the distributional
comparisons, which we display once.

\begin{table}[!htbp]
\caption{The three families of chains on $V=\{1,\dots,7\}$, with $30$
chains each.}\label{tab:families}
\centering
\small
\begin{tabular}{@{}l>{\raggedright\arraybackslash}p{8.2cm}c@{}}
\toprule
family & transition matrix & reversible\\
\midrule
dense & $P_{uv}=W_{uv}/\sum_{x}W_{ux}$, with $W_{uv}$ independent and
uniform on $\{1,\dots,5\}$ & in general no\\
\addlinespace[2pt]
graph & simple random walk on a random graph, in which each of the $21$
possible edges is present independently with probability $0.45$, before
conditioning on connectivity & yes\\
\addlinespace[2pt]
drift & $P_{uv}\propto W_{uv}$, with $W_{u,u+1}=3$ and $W_{u,u-1}=1$
(indices modulo $7$), and $W_{uu}$ uniform on $\{0,1,2\}$ & no\\
\bottomrule
\end{tabular}
\end{table}

For each chain and root set, we enumerate the parent maps in
$\Fset_B$. We compute the order masses from the principal cofactors,
then obtain the stopped-walk law $q_A$ with seed $1$ from
\eqref{eq:closed}. The comparison law is $q_W\propto w$. We measure
their difference by total variation distance,
$\tfrac12\sum_f|q_A(f)-q_W(f)|$. The stationary distributions and
normalizing constants are computed in exact rational arithmetic, while the remaining linear algebra uses double precision.

To measure support loss, we must distinguish structural zeros from
numerical roundoff. By Remark~\ref{rem:contract}, $\varphi_1(F_0)=0$ exactly when some
root cannot reach $1$ through positive entries of $Q$. The script
\nolinkurl{zero_check.py} compares this combinatorial test
with a $10^{-14}$ threshold on the factor
$g(B,f)^{\top}L[B]^{-1}e_1$, which differs from the completion weight
by a positive constant for each fixed chain and root set. The tests
agree in all $546{,}120$ cases. Computed values of this factor at
structural zeros are at most $2.2\times10^{-16}$ in magnitude, and
the smallest positive value is $2.7\times10^{-2}$.

The script \nolinkurl{make_experiments.py} generates the
chains and enumerated laws with random seed $20260918$. The figure
script \nolinkurl{plot_bias.py} reads these stored
outputs to produce Figure~\ref{fig:bias}, without recomputing the laws.
All scripts named in this section and their stored outputs are in the
repository \url{https://github.com/YuZh98/aldous-broder-numerics}.

\subsection{Results}

Figure~\ref{fig:bias} shows the total variation distance by root count.
For one root, the two laws agree, as Theorem~\ref{thm:ab} requires.
With two roots, the median total variation distance is about $0.26$
for both the dense and graph families. With three roots, the medians
are about $0.27$ and $0.30$, respectively. The largest distance is
$0.75$, attained by one graph chain with four roots. Beyond three
roots, the median distances decrease to about $0.11$ at $|B|=6$.
These examples show that a stopped walk can differ substantially
from the target forest law even when the underlying chain is
reversible, as in the graph family.

The drift example has smaller distances for the chosen roots and
seed, increasing from $0.001$ at $|B|=2$ to $0.083$ at $|B|=6$.
This behavior depends on the contiguous roots and seed $1$, as the
next experiment shows.

\subsection{Sensitivity}\label{sec:sens}
We vary the placement of the roots, the strength of the drift and the
connectivity of the graph. The script
\nolinkurl{robustness.py} reproduces the values below.

\emph{Placement.} For the same $30$ dense and $30$ graph chains, we
replace the nested roots and the seed $1$ by a uniformly random root set
of each size and a uniformly random seed in it. The median distances for
$|B|=2,\dots,6$ are $0.24$, $0.29$, $0.21$, $0.20$ and $0.13$ for the
dense family and $0.23$, $0.28$, $0.30$, $0.27$ and $0.16$ for the graph
family. In these samples, the medians are close to those of
Figure~\ref{fig:bias}.

\emph{Drift.} For the drift cycle, we examine the effect of root
placement and seed choice. Table~\ref{tab:drift} gives the range of the distance over
the seeds in the root set, for contiguous roots $\{1,\dots,b\}$ and for
roots spread evenly around the cycle, and for three drift strengths.
The small distances of Figure~\ref{fig:bias} occur only for the contiguous
roots and one seed. With another seed, or with spread roots, the
distance is of the same order as for the other families.

\emph{Connectivity.} For $30$ random connected graphs with edge
probability $0.30$, $0.45$ and $0.70$, the median distance for two roots
is $0.40$, $0.23$ and $0.23$, and the fraction of chains with lost mass
is $0.57$, $0.23$ and $0.13$. In these samples, the sparser graphs
have larger median distances and a higher frequency of support loss.

\begin{table}[!htbp]
\caption{Range of the total variation distance over the seeds in the
root set, for the drift cycle with forward weight $a$ and backward
weight $1$.}\label{tab:drift}
\centering
\footnotesize
\begin{tabular}{@{}llccccc@{}}
\toprule
$a$ & roots & $|B|=2$ & $3$ & $4$ & $5$ & $6$\\
\midrule
$1.5$ & contiguous & $0.03$--$0.23$ & $0.07$--$0.36$ & $0.13$--$0.44$ & $0.16$--$0.45$ & $0.01$--$0.37$\\
 & spread & $0.46$ & $0.56$ & $0.48$--$0.51$ & $0.26$--$0.39$ & $0.01$--$0.37$\\
\addlinespace[2pt]
$3$ & contiguous & $0.00$--$0.22$ & $0.00$--$0.29$ & $0.01$--$0.31$ & $0.03$--$0.30$ & $0.08$--$0.25$\\
 & spread & $0.30$--$0.31$ & $0.43$--$0.44$ & $0.40$--$0.47$ & $0.24$--$0.40$ & $0.08$--$0.25$\\
\addlinespace[2pt]
$6$ & contiguous & $0.00$--$0.14$ & $0.00$--$0.16$ & $0.00$--$0.17$ & $0.01$--$0.16$ & $0.02$--$0.14$\\
 & spread & $0.16$ & $0.27$--$0.28$ & $0.26$--$0.34$ & $0.14$--$0.26$ & $0.02$--$0.14$\\
\bottomrule
\end{tabular}
\end{table}

\subsection{Lost mass, effective sample size and an estimator}
Total variation distance alone does not identify whether correction
is possible. We therefore report two diagnostics in Table~\ref{tab:lost}:
the unreachable target mass $q_W(S_1^c)$ and the limiting normalized
effective sample size of the inverse-completion weights. The former
measures the part of the target law that no reweighting can recover.
The latter measures weight dispersion on the reachable forests.
For $m$ independent draws $F^{(i)}\sim q_A$ with
weights $a_i:=1/\varphi_1(F^{(i)})$, the normalized effective sample size
$(\sum_ia_i)^2/(m\sum_ia_i^2)$ converges almost surely to
\[
\mathrm{ESS}/m:=\frac{(\E_{q_A}a)^2}{\E_{q_A}a^2}
=\Big(\sum_{F_0\in S_1}\frac{\bar q_W(F_0)^2}{q_A(F_0)}\Big)^{-1},
\qquad \bar q_W:=q_W(\cdot\mid S_1),
\]
which is the value reported in the table. The precision of a particular estimate also
depends on the function $h$ and its association with the weights.

Lost mass is zero for all dense and drift chains. In the graph
family, it is positive for about half the chains with two or three
roots and reaches $0.75$. For the graph chain with four roots that
attains the largest total variation distance in Figure~\ref{fig:bias}, the
distance is entirely lost mass. Reweighting cannot reduce that
distance. On the reachable forests, the median limiting normalized
ESS is between $0.59$ and $0.95$ for the dense and graph families,
with a minimum of $0.31$.

\begin{table}[!htbp]
\caption{Lost target mass $q_W(S_1^c)$ for the graph family and
limiting normalized effective sample size $\mathrm{ESS}/m$ for the three
families, over the $30$ chains of each family. The lost mass is zero
for every dense and drift chain.}\label{tab:lost}
\centering
\footnotesize
\begin{tabular}{@{}ccccccccc@{}}
\toprule
& \multicolumn{3}{c}{graph: lost mass} & \multicolumn{5}{c}{$\mathrm{ESS}/m$}\\
\cmidrule(lr){2-4}\cmidrule(l){5-9}
$|B|$ & fraction $>0$ & median & max & dense med. & dense min & graph med. & graph min & drift\\
\midrule
$2$ & $0.53$ & $0.15$ & $0.58$ & $0.59$ & $0.31$ & $0.83$ & $0.69$ & $1.00$\\
$3$ & $0.57$ & $0.10$ & $0.58$ & $0.68$ & $0.43$ & $0.73$ & $0.52$ & $1.00$\\
$4$ & $0.37$ & $0$ & $0.75$ & $0.79$ & $0.59$ & $0.76$ & $0.34$ & $1.00$\\
$5$ & $0.20$ & $0$ & $0.50$ & $0.86$ & $0.72$ & $0.87$ & $0.47$ & $0.99$\\
$6$ & $0.03$ & $0$ & $0.20$ & $0.94$ & $0.83$ & $0.95$ & $0.62$ & $0.97$\\
\bottomrule
\end{tabular}
\end{table}

To illustrate the effect on an estimate, consider the expected size
of the seed's tree, $h(F_0):=1+|D_1|$, with $B=\{1,2\}$.
The first dense chain has complete support. Its target expectation
is $3.237$, while the uncorrected stopped-walk expectation is $4.359$.
The self-normalized correction converges to $3.237$, although the
sample ratio is generally biased at finite sample sizes. Across
$4{,}000$ replications with $m=200$ independent draws each, the
corrected estimate has mean $3.241$ and standard deviation $0.175$.
The limiting normalized ESS is $0.68$.

For the graph chain with the largest lost mass, $0.577$, the target
expectation is $1.923$ and the uncorrected expectation is $3.733$.
The correction converges to the conditional target $3.182$. Thus
correct weighting on the sampled forests cannot recover the forests
that are never sampled. These calculations are generated by
\nolinkurl{estimator_demo.py}, which reuses the chains from
\nolinkurl{make_experiments.py} and uses random seed $1$ for the
replications.

\subsection{Checks}
We check the computations in two independent ways. First, for every
chain, root set, and seed, the computed stopped-walk probabilities sum
to one and agree with the tree-law marginal obtained by enumerating the
completed arborescences. The largest discrepancy is
$4.8\times10^{-15}$.

Second, we run Algorithm~\ref{alg:stopped} $200{,}000$ times on the first dense
chain with $|B|=2$ and seed $1$. This simulation uses only the forward
transition matrix. Of the $4{,}802$ expected counts, $755$ are below
$5$ and $79$ below $1$, and the smallest is $0.26$. We therefore calibrate
the goodness-of-fit statistics with $2{,}000$ multinomial samples of
the same size from $q_A$, rather than a chi-squared approximation.
The reported upper-tail $p$-values have Monte Carlo standard errors
at most $0.011$.

The Pearson statistic is $4{,}818.1$, with simulated $p$-value $0.41$.
The maximum standardized deviation across forests is $3.52$, with
upper-tail $p$-value $0.995$, so this maximum lies near the lower
$0.5\%$ tail of the simulated reference distribution. Neither
upper-tail test rejects $q_A$. Overall, the forward-walk simulation
agrees with the law computed from \eqref{eq:closed}.

\bibliographystyle{plainnat}
\bibliography{refs}
\end{document}